\documentclass{amsart}

\def\R{\mathbb R}
\def\C{\mathbb C}
\def\Z{\mathbb Z}
\def\Q{\mathbb Q}

\newtheorem{crl}{Corollary}[section]

\newtheorem{lmm}{Lemma}[section]

\newtheorem{prp}{Proposition}[section]

\newtheorem{thm}{Theorem}[section]
\newtheorem{sublmm}{Sublemma}[section]

\theoremstyle{definition}
\newtheorem{rem}{Remark}
\newtheorem{assump}{Assumption}
\newtheorem{dfn}{Definition}[section]

\newtheorem{conj}{Conjecture}[section]
\newtheorem{prob}{Problem}[section]
\theoremstyle{remark}

\title{Counting pseudo-holomorphic discs \\
in 
Calabi-Yau 3 fold
}
\author{Kenji Fukaya}

\begin{document}

\maketitle

\begin{abstract}   
In this paper we define an invariant of a pair of 6 dimensional symplectic                    
manifold with vanishing 1st Chern class and its Lagrangian submanifold 
with 
vanishing Maslov index.
This invariant is a function on the set of the path connected components of 
the bounding cochains (solution of 
A infinity version of Maurer-Cartan equation of the filtered A infinity algebra 
associated to the Lagrangian submanifold). 
In the case when the Lagrangian submanifold is a rational homology sphere, 
it becomes a numerical invariant.
\par
This invariant depends on the choice of almost complex structure.
The way how it depends on the almost complex structure
is described by a wall crossing formula which involves 
moduli space of pseudo-holomorphic spheres.
\end{abstract}

\tableofcontents                        

\section{Introduction}

\label{introsec}

\footnote{Supported by JSPS Grant-in-Aid for Scientifique Research No. 19104001 and 
Global COE Program G08.}This paper is a continuation of \cite{FOOO080} Subsection 3.6.4 and 
\cite{F1}. 
\par
Let $(M,\omega)$ be a symplectic manifold of (real) dimension $2\times 3$. 
We assume that $c^1(M) = 0$ in $H^2(M;\Q)$.
(Here we use compatible almost complex structure of tangent bundle to define $c^1(M)$.)
Let $L \subset M$ be a relatively spin Lagrangian submanifold and
$\mu_L : H_2(M,L;\Z) \to 2\Z$ its Maslov index homomorphism. (See \cite{FOOO080} Subsection 2.1.1.)
We assume that $\mu_L$ is $0$. 
In this paper we consider such a pair $(M,L)$.
A typical example is a pair of a
Calabi-Yau 3 fold $M$, and its 
special Lagrangian submanifold  $L$.
This is one of the most interesting cases of (homological) mirror symmetry.
Our main purpose of this paper is to define and study an invariant of 
such $(M,L)$. It is independent of various choices involved in the 
construction but depends on the almost complex structure $J$ of $M$.
\par
We consider $\mathcal M(L;J;\Lambda_+)$ the set of `$\Lambda_+$-valued points of 
Maurer-Cartan formal scheme'  of the filtered $A_{\infty}$ structure associated to $L$.
This is the set of gauge equivalence classes of bounding cochains 
and defined in \cite{FOOO080} Section 4.3. (Here we include $J$ in the notation 
since $J$ dependence is rather crucial in this paper.)
We study cyclic filtered $A_{\infty}$ algebra $(\Lambda(L),\langle \cdot \rangle,
\{\frak m_{k,\beta}^J\})$ produced in \cite{F1} by modifying the construction of \cite{FOOO080}.
In our case where $\mu_L$ is $0$ we can reduce the coefficient ring to 
$\Lambda_0 = \Lambda_{0,nov}^{(0)}$, that is the degree $0$ part of the universal 
Novikov ring with $\R$ coefficient.
(The universal Novikov ring is defined at \cite{FOOO080} begining of Subsection 1.2.)
We denote by $\Lambda_+$ its maximal ideal.
Let $[b] \in \mathcal M(L;J;\Lambda_+)$. We define a superpotential (without leading term) by:
\begin{equation}\label{superpote}
\Psi'(b;J) = \sum_{k=0}^{\infty}\sum_{\beta} 
\frac{T^{\beta\cap \omega}}{k+1}\langle \frak m_{k,\beta}^J(b,\ldots,b),b \rangle.
\end{equation}
To obtain a superpotential which is independent of the 
perturbation and other choices involved, we need to add the constant term to it.
Note $ \frak m^J_{k,\beta}$ is defined by using moduli space $\mathcal M_{k+1}(\beta;J)$ of 
pseudo-holomorphic discs with $k+1$ marked points and of homology class 
$\beta \in H_2(M,L;\Z)$. We use $\mathcal M_{0}(\beta;J)$, the moduli space of $J$ holomorphic discs,
of homology class $\beta$ without marked point, to define 
\begin{equation}
\frak m^{J}_{-1,\beta}\,\, \text{``=''} \,\, \# \mathcal M_{0}(\beta;J).
\end{equation}
(See Sections \ref{pseudoiso},\ref{geomsec} for precise definition.) and put
\begin{equation}\label{correctedsup}
\Psi(b;J) = \Psi'(b;J) + \sum_{\beta\in H_2(L;\Z)}T^{\beta\cap\omega}\frak m^{J}_{-1,\beta}.
\end{equation}
More precisely we assume that our almost complex structure $J$ satisfies the following:
\begin{assump}\label{assumpJ}
There exists no nontrival $J$-holomorphic sphere $v : S^2 \to M$ such that 
$v(S^2) \cap L \ne \emptyset$.
\end{assump}
By dimension counting we find that the set of such $J$ is dense.
\begin{thm}\label{superpotential}
\begin{enumerate}
\item If $J$ satisfies the Assumption \ref{assumpJ}, then there exists 
a function 
$$
\Psi : H^1(L;\Lambda_{+}) \to \Lambda_{+}
$$
which depends not only on $J$ but also on perturbation etc.
\item
There exists an isomorphism between the set
$\mathcal M(L;J;\Lambda_+)$ and the set of critical points of $\Psi$.
\item
The restriction of $\Psi$ to its critical point set $\mathcal M(L;J;\Lambda_+)$ 
depends only on $M,L,J$ and is independent of the choice of perturbation etc.
\end{enumerate}
\end{thm}
We call $\Psi$ the {\it superpotential}. 
The value $\Psi(b)$ depends only on the path connected component 
of $b \in \mathcal M(L;J)$. See Proposition \ref{dependonlyoncomp}.
\begin{crl}\label{maincor}
If $L$ is a rational homology sphere in addition, then $\mathcal M(L;J;\Lambda_+)$ is one point.
So the value of $\Psi$ at that point is an invariant of $M,L,J$.
\end{crl}
In Sections \ref{superpotsec} and \ref{pseudoiso} we develop the theory 
of superpotential of cyclic filtered $A_{\infty}$ algebra of dimension 3 with additional data 
corresponding to $\frak m^{J}_{-1,\beta}$. 
In Section \ref{superpotsec} we fix our cyclic filtered $A_{\infty}$ algebra and 
review the construction of superpotential and its gauge invariance. 
We next study its relation to pseudo-isotopy of cyclic filtered $A_{\infty}$ algebra
to complete the algebraic part of the proof of Theorem \ref{superpotential} in Section  \ref{pseudoiso}.
The algebraic structure we assumed in Sections \ref{superpotsec} and \ref{pseudoiso} are 
realized in Section \ref{geomsec}, where the proof of Theorem \ref{superpotential} 
is completed.
\par
We can extend the domain $H^1(L;\Lambda_{+})$ of the definition of $\Psi(b;J)$ as follows.
Let $\text{\bf e}_i$, $i=1,\ldots,b_1$ be the basis of $H^1(L;\Z)/{\rm Torsion}$.
We put $\text{\bf b} = \sum x_i \text{\bf e}_i$ where $x_i \in \Lambda_0$.
We put 
$$
y_i = e^{x_i} = \sum_{k=0}^{\infty} \frac{1}{k!} x_i^k.
$$
We define the {\it strongly convergent Laurent power series ring} 
(See \cite{BGR84}.)
$$\Lambda_0 \langle\!\langle y_1,\ldots,y_{b_1},y_1^{-1},\ldots,y_{b_1}^{-1}\rangle\!\rangle$$
as the set of formal sums
\begin{equation}\label{stconele}
f(y_1,\ldots,y_{b_1}) = \sum_{i=1}^{\infty} T^{\lambda_i} P_i(y_1,\ldots,y_{b_1})
\end{equation}
where $\lambda_i \in \R_{\ge 0}$ with $\lim_{i\to\infty} \lambda_i = \infty$ and 
$P_i$ are Laurent {\it polynomial}.
We remark that for each $f$ as in (\ref{stconele}) and 
$\frak y_1,\ldots,\frak y_{b_1} \in \Lambda_0$ with $v(\frak y_i) =0$, the sum
$$
\sum_{i=1}^{\infty} T^{\lambda_i} P_i(\frak y_1,\ldots,\frak y_{b_1} )
$$
converges in $T$ adic topology. Therefore $f(\frak y_1,\ldots,\frak y_{b_1})$ is well defined.
\begin{thm}\label{convmain}
\begin{enumerate}
\item
$\Psi(b,J) \in \Lambda_0 \langle\!\langle y_1,\ldots,y_{b_1},y_1^{-1},\ldots,y_{b_1}^{-1}\rangle\!\rangle.$
\item
There exists $\delta > 0$ such that $\Psi$ is extended to 
\begin{equation}
\{(y_1,\ldots,y_{b_1}) \mid -\delta < v(y_i) < \delta\}.
\end{equation}
\item
Its critical point set is identified with $\mathcal M(L;J)_{\delta}$ which is  introduced in Theorem 
{\rm 1.2} {\rm\cite{F1}}.
\item
The restriction of $\Psi$ to $\mathcal M(L;J)_{\delta}$ is independent of the perturbation etc. and 
depends only on $M,L,J$.
\end{enumerate}
\end{thm}
Here $v(\cdot)$ is defined by
$$
v\left(\sum a_iT^{\lambda_i}\right) = \inf\{ \lambda_i \mid a_i \ne 0\}.
$$
\par
We prove Theorem \ref{convmain} in Section \ref{conv}.
\par
In Section \ref{canonicalmodelsec} we use canonical model constructed in 
\cite{FOOO080} Section 4.5 and \cite{F1} Section 10, 
to rewrite the definition of $\Psi$. 
\par
In Section \ref{wallcrosssec} we discuss the way how superpotential $\Psi$ depends on 
almost complex structure. The main result is 
Theorem \ref{wallcross} below.
We assume that $J_0$ and $J_1$ satisfy Assumption \ref{assumpJ}.
We take a path $\mathcal J = \{J_t\mid t \in [0,1]\}$ 
of tame almost complex structures joining them.
Let $\mathcal M_1^{\text{\rm cl}}(\alpha;J)$ be the moduli space of $J$ holomorphic stable maps 
of genus zero in $M$ of homology class $\alpha \in H_2(M;\Z)$ and with one marked point.
It has a Kuranishi structure of (virtual) dimension 2. We put
\begin{equation}
\mathcal M^{\text{\rm cl}}_1(\alpha;\mathcal J) = \bigcup_{t \in [0,1]} \{t\} \times \mathcal M^{\text{\rm cl}}_1(\alpha;J_t).
\end{equation}
Using evaluation map $ev : \mathcal M^{\text{\rm cl}}_1(\alpha;J) \to M$ we obtain a virtual 
fundamental chain 
$ev_*([\mathcal M^{\text{\rm cl}}_1(\alpha;\mathcal J)])$
of dimension $3$. Since $J_0$ and $J_1$ satisfy Assumption \ref{assumpJ} it follows that
$$
L \cap  ev(\partial\mathcal M^{\text{\rm cl}}_1(\alpha;\mathcal J)) = \emptyset.
$$
Therefore
\begin{equation}\label{wallcrossformula}
n(L;\alpha;\mathcal J) = [L] \cap ev_*([\mathcal M^{\text{\rm cl}}_1(\alpha;\mathcal J)]) \in \Q
\end{equation}
is well-defined. Moreover it depends only on $M,L,\alpha,J_0,J_1$ and 
is independent of the path $\mathcal J$.
\begin{thm}\label{wallcross}
Let $[b] \in \mathcal M(L;J_0)$.
We take canonical isomorphism $I_* : \mathcal M(L;J_0) \to \mathcal M(L;J_1)$ in 
{\rm \cite{FOOO080}} Section {\rm 4.3}.  Then we have:
\begin{equation}\label{wallcrossformula}
\Psi(I_*(b),J_1) - \Psi(b,J_0)
= \sum_{\alpha \in H_2(M;\Z)} T^{\alpha \cap \omega} n(L;\alpha;\mathcal J).
\end{equation}
\end{thm}
Theorem \ref{wallcross} is proved in Section \ref{wallcrosssec}.
In Section \ref{DTawx} we discuss some conjectures, open problems, and relations 
to various related topics.
\begin{rem}
\begin{enumerate}
\item 
Superpotential of the form (\ref{superpote}) appears in 
the physics literature \cite{Laza01, Tom01}.
\item
The idea to include the 2nd term of (\ref{correctedsup}) to obtain 
a numerical invariant of Lagrangian submanifold is due to 
D. Joyce. It was communicated to the author by P. Seidel around 
2002, who also explained him the importance of cyclic symmetry for this 
purpose. 
(However the appearance of nontrivial wall crossing by the change of $J$  
was unknown at that time.)
\item
The appearance of the nonzero wall crossing term 
in the right hand side of (\ref{wallcrossformula}) is 
closely related to the phenomenon discussed in \cite{FOOO080} Section 3.8 and Subsection 7.4.1.
Around the same time as the authors of \cite{FOOO080} found this phenomenon, a similar 
observation was done independently by M. Liu \cite{Liu02}. 
\item
A related homological algebra was discussed before 
by \cite{Cho, Ka}. The part concerning the second term of  (\ref{correctedsup}) 
is  not discussed there. 
\item
All the $A_{\infty}$ algebras and pseudo-isotopies between them 
which appear in the geometric situation in this paper, 
are unital. We omit the argument on unitality since it is a straight forward analog 
of one in \cite{F1}.
\end{enumerate}
\end{rem}
The author would like to thank to Y.-G.Oh, H. Ohta, and K. Ono. 
Joint works with them are indispensable for the author to write this paper.
\section{Superpotential and its gauge invariance}
\label{superpotsec}
Let $(C,\langle\cdot\rangle,\{\frak m_{k,\beta}\})$ be a $G$-gapped cyclic filtered $A_{\infty}$
algebra of dimension $3$. 
Recall that $G \subset \R_{\ge 0} \times 2\Z$ is a discrete submononid in the sense of 
\cite{FOOO080} Condition 3.1.6, \cite{F1} Definition 6.2.
In this paper we always assume 
\begin{equation}
G \subset \R_{\ge 0} \times \{0\}.
\end{equation}
Namely $G \subset \R_{\ge 0}$. In this case 
$$
\frak m_{k,\beta} : B_k(\overline C[1]) \to \overline C[1]
$$ 
is always of degree $1$ (after degree shift).
We put $C_+ = \overline C\otimes_{\R}\Lambda_+$.
\begin{dfn}
We define 
$$
\Psi' : C_+^1 \to \Lambda_0
$$ by
\begin{equation}\label{superpoteprime}
\Psi'(b) = \sum_{k=0}^{\infty}\sum_{\beta\in G}
\frac{T^{E(\beta)}}{k+1}\langle \frak m_{k,\beta}(b,\ldots,b),b\rangle.
\end{equation}
\end{dfn}
\begin{rem}
\begin{enumerate}
\item 
More precisely the right hand side of (\ref{superpoteprime}) converges in 
$T$ adic topology.
In various cases, 
it converges in the topology of  \cite{F1} Definition 13.1.
(It converges in the case of filtered $A_{\infty}$ algebra 
of Lagrangian Floer theory by  \cite{F1} Theorem 1.2.)
See Section \ref{conv} on the convergence.
\item 
Since $\deg'b = \deg b -1 = 0$. We have
$$
\deg'  \frak m_{k,\beta}(b,\ldots,b) = 1.
$$
Namely $\deg  \frak m_{k,\beta}(b,\ldots,b) +\deg b = 3$.
Therefore in case the dimension of our cyclic filtered $A_{\infty}$
algebra is 3, the inner product in the right hand side of (\ref{superpoteprime})
is well defined. 
\end{enumerate}
\end{rem}
We fix a basis $\text{\bf e}_i \in \overline C$ and 
put
$
b = \sum x_i\text{\bf e}_i.
$
Then
$
\Psi'(b) = \sum_{\beta} P_{\beta}(x_1,\ldots)
$
where $P_{\beta}$ is a formal power series.
Therefore we can differentiate $\Psi'$ formally.
We have:
\begin{prp}\label{prp21}
If $b \in C^1_+$ then the differential of 
$\Psi'$ vanishes at $b$ if and only if
\begin{equation}\label{MCeq}
\sum_{\beta\in H_2(M,L;\Z)}\sum_{k=0}^{\infty} T^{E(\beta)} \frak m_{k,\beta}(b,\ldots,b) = 0.
\end{equation}
\end{prp}
This is \cite{FOOO080} Proposition 3.6.50.
(\ref{MCeq}) is called the $A_{\infty}$ Maurer-Cartan equation.
\begin{dfn}
$\widetilde{\mathcal M}(C;\Lambda_+)$ is the set of all $b \in C^1_+$ satisfying 
(\ref{MCeq}).
\end{dfn}
We next review the definition of gauge equivalence from \cite{FOOO080} Section 4.3.
We consider 
\begin{equation}\label{bcdef}
b(t) = \sum_{\beta: E(\beta)>0} T^{E(\beta)}b_{\beta}(t), \qquad
c(t) = \sum_{\beta: E(\beta)>0} T^{E(\beta)}c_{\beta}(t)
\end{equation}
where 
$b_{\beta}(t)$, $c_{\beta}(t)$ are polynomial with coefficeint in 
$\overline C^1$, $\overline C^0$ respectively. 
\begin{dfn}[See \cite{FOOO080} Proposition 4.3.5]
We say $b_0 \in \widetilde{\mathcal M}(C;\Lambda_+)$ is 
gauge equivalent to $b_1 \in \widetilde{\mathcal M}(C;\Lambda_+)$ 
if there exists $b(t), c(t)$ as in (\ref{bcdef}) such that:
\smallskip
\begin{enumerate}
\item $b(0) = b_0$, $b(1) = b_1$.
\item 
\begin{equation}
\frac{d}{dt} b(t)  + 
\sum_{k=1}^{\infty} \frak m_k(b(t),\ldots,b(t),c(t),b(t),\ldots,b(t)) = 0.
\end{equation}
\end{enumerate}
\item It is proved in \cite{FOOO080} Lemma 4.3.4 that gauge equivalence is 
an equivalence relation. We denote by
${\mathcal M}(C;\Lambda_+)$ the set of gauge equivalence classes.
\end{dfn}
\begin{rem}
It follows from 1,2 that $b(t) \in  \widetilde{\mathcal M}(C;\Lambda_+)$
for any $t$. (\cite{FOOO080} Lemma 4.3.7.)
\end{rem}
\begin{prp}\label{gginvariance}
If $b_0 \in \widetilde{\mathcal M}(C;\Lambda_+)$ is 
gauge equivalent to $b_1 \in \widetilde{\mathcal M}(C;\Lambda_+)$  then
$$
\Psi'(b_0) = \Psi'(b_1).
$$
\end{prp}
\begin{proof}
We have
\begin{equation}\label{keisanpsi}
\aligned
\frac{d}{dt}\Psi'(b) &=\frac{d}{dt} \sum_{k=0}^{\infty}\sum_{\beta\in G}
\frac{T^{E(\beta)}}{k+1}\langle \frak m_{k,\beta}(b(t),\ldots,b(t)),b(t)\rangle
\\
&= \sum_{k=0}^{\infty}\sum_{\beta\in G}
\frac{T^{E(\beta)}}{k+1}\left\langle \frak m_{k,\beta}(b(t),\ldots,\frac{db(t)}{dt},\ldots,b(t)),b(t)\right\rangle \\
&\quad+ \sum_{k=0}^{\infty}\sum_{\beta\in G}
\frac{T^{E(\beta)}}{k+1}\left\langle \frak m_{k,\beta}(b(t),\ldots,b(t)),\frac{db(t)}{dt}\right\rangle \\
&=  \sum_{k=0}^{\infty}\sum_{\beta\in G}\left\langle \frak m_{k,\beta}(b(t),\ldots,b(t)),\frac{db(t)}{dt}\right\rangle.
\endaligned
\end{equation}
Since $b(t) \in \widetilde{\mathcal M}(C;\Lambda_+)$, it follows that (\ref{keisanpsi}) is zero.
\end{proof}
By Proposition \ref{gginvariance} we obtain 
\begin{equation}
\Psi' : {\mathcal M}(C;\Lambda_+) \to \Lambda_0.
\end{equation}
We remark that in the proof of Proposition \ref{gginvariance} we only use the existence of 
families $b(t)$ in $\widetilde{\mathcal M}(C;\Lambda_+)$ joining $b_0$ and $b_1$. 
In other words, we did not use the existence of $c(t)$.
Therefore we have:
\begin{prp}\label{dependonlyoncomp}
If the map $t \mapsto b(t) \in \widetilde{\mathcal M}(C;\Lambda_+)$ is a $C^1$ map then 
$$
\Psi(b(0)) = \Psi(b(1)).
$$
\end{prp}
\begin{rem}
Proposition \ref{dependonlyoncomp} may imply that superpotential is locally constant 
on $\mathcal M(C;\Lambda_+)$ and so $\Psi$ depends only on the `irreducible component' 
of $\mathcal M(C;\Lambda_+)$. Since the property of $\mathcal M(C;\Lambda_+)$ as a topological 
space can be rather complicated, we do not try to study this point in this paper.
\end{rem}
\section{Pseudo-isotopy invariance}
\label{pseudoiso}
In \cite{F1} Definition 8.5, it is defined that
 $(C,\langle \cdot \rangle,\{\frak m^t_{k,\beta}\},\{\frak c^t_{k,\beta}\})$
is a pseudo-isotopy of cyclic filtered $A_{\infty}$ algebra if:
\smallskip
\begin{enumerate}
\item  $\frak m^t_{k,\beta}$ and $\frak c^t_{k,\beta}$ are smooth. Namely 
$$
t\mapsto \frak m^t_{k,\beta}(x_1,\ldots,x_k)
$$
is smooth. (That is the coefficient is a smooth function of $t \in [0,1]$.)
\item For each (but fixed) $t$, the triple  $(C,\langle \cdot \rangle,\{\frak m^t_{k,\beta}\})$ defines a cyclic fitered $A_{\infty}$
algebra.
\item For each (but fixed) $t$, and $x_i \in \overline C[1]$, we have
\begin{equation}\label{inhomomainformula}
\langle 
\frak c^t_{k,\beta}(x_1,\ldots,x_k),x_0
\rangle
= 
(-1)^*\langle 
\frak c^t_{k,\beta}(x_0,x_1,\ldots,x_{k-1}),x_{k} 
\rangle
\end{equation}
$* = (\deg \text{\bf x}_0+1)(\deg \text{\bf x}_1+\ldots + \deg \text{\bf x}_k + k)$.
\item  For each $x_i \in \overline C[1]$
\begin{equation}\label{isotopymaineq}
\aligned
&\frac{d}{dt} \frak m_{k,\beta}^t(x_1,\ldots,x_k) \\
&+ \sum_{k_1+k_2=k}\sum_{\beta_1+\beta_2=\beta}\sum_{i=1}^{k-k_2+1}
(-1)^{*}\frak c^t_{k_1,\beta_1}(x_1,\ldots, \frak m_{k_2,\beta_2}^t(x_i,\ldots),\ldots,x_k) \\
&- \sum_{k_1+k_2=k}\sum_{\beta_1+\beta_2=\beta}\sum_{i=1}^{k-k_2+1}
\frak m^t_{k_1,\beta_1}(x_1,\ldots, \frak c_{k_2,\beta_2}^t(x_i,\ldots),\ldots,x_k)\\
&=0.
\endaligned
\end{equation}
Here $* = \deg' x_1 + \ldots + \deg'x_{i-1}$.
\item 
$\frak m_{k,(0,0)}^t$  is independent of $t$. $\frak c_{k,(0,0)}^t = 0$.
\end{enumerate}
\begin{dfn}
\begin{enumerate}
\item
$(C,\langle \cdot \rangle,\{\frak m_{k,\beta}\},\{\frak m_{-1,\beta}\})$
is said to be an {\it inhomogeneous cyclic filtered $A_{\infty}$ algebra} if 
$(C,\langle \cdot \rangle,\{\frak m_{k,\beta}\})$ is 
cyclic filtered $A_{\infty}$ algebra and $\frak m_{-1,\beta} \in \R$.
\item
$(C,\langle \cdot \rangle,\{\frak m^t_{k,\beta}\},\{\frak c^t_{k,\beta}\},\{\frak m^t_{-1,\beta}\})$
is said to be a {\it pseudo-isotopy of inhomogeneous cyclic filtered $A_{\infty}$ algebra} 
if $(C,\langle \cdot \rangle,\{\frak m^t_{k,\beta}\},\{\frak c^t_{k,\beta}\})$ is 
a pseudo-isotopy of cyclic filtered $A_{\infty}$ algebra, 
$$
t \mapsto \frak m^t_{-1,\beta}
$$
is a real valued smooth function and if
\begin{equation}\label{dervativemt}
\frac{d}{dt}\frak m^t_{-1,\beta} 
+ \sum_{\beta_1+\beta_2=\beta}\langle \frak c^t_{0,\beta_1}(1),
\frak m^t_{0,\beta_2}(1)\rangle = 0.
\end{equation}
\end{enumerate}
\end{dfn}
Let $(C,\langle \cdot \rangle,\{\frak m^t_{k,\beta}\},\{\frak c^t_{k,\beta}\})$ be 
a pseudo-isotopy of cyclic filtered $A_{\infty}$ algebra.
We consider cyclic filtered $A_{\infty}$ algebras $(C,\langle \cdot \rangle,\{\frak m^0_{k,\beta}\})$
and $(C,\langle \cdot \rangle,\{\frak m^1_{k,\beta}\})$.
By  \cite{F1} Theorem 8.2 there exists an 
isomorphism
\begin{equation}
\frak c = \frak c(1;0) : (C,\langle \cdot \rangle,\{\frak m^0_{k,\beta}\}) 
\to (C,\langle \cdot \rangle,\{\frak m^1_{k,\beta}\})
\end{equation}
of cyclic filtered $A_{\infty}$ algebra.
It induces
$$
\frak c_* : \mathcal M (C,\{\frak m^0_{k,\beta}\}) 
\to \mathcal M (C,\{\frak m^1_{k,\beta}\})
$$
by \cite{FOOO080} Theorem 4.3.22.
The main result of this section is as follows.
\begin{thm}\label{pseudoisowelldef}
We have
\begin{equation}
\Psi'(\frak c_*(b)) + \sum_{\beta}T^{E(\beta)}\frak m^1_{-1,\beta}
= \Psi'(b) + \sum_{\beta}T^{E(\beta)}\frak m^0_{-1,\beta}.
\end{equation}
\end{thm}
\begin{proof}
We also constructed 
$$
\frak c(t;0) : (C,\langle \cdot \rangle,\{\frak m^0_{k,\beta}\}) 
\to (C,\langle \cdot \rangle,\{\frak m^t_{k,\beta}\})
$$
in \cite{F1} Definition 9.4. It is an isomorphism and depends smoothly on $t$.
We put 
$$
b(t) = \frak c(t;0)_*(b) = \sum_{k,\beta}  \frak c_{k,\beta}(t;0)(b,\ldots,b)
$$
and
\begin{equation}\label{deff(t)}
\aligned
f(t) &= \Psi'(b(t)) + \sum_{\beta}T^{E(\beta)}\frak m^t_{-1,\beta} \\
&= \sum_{k=0}^{\infty}\frac{1}{k+1}
\langle \frak m^t_{k}(b(t),\ldots,b(t)),b(t)\rangle 
+ \sum_{\beta}T^{E(\beta)}\frak m^t_{-1,\beta}.\endaligned
\end{equation}
We calculate the derivative of $f(t)$. The 
derivative of the first term is:
\begin{equation}\label{keisan1kou}
\aligned
&\sum_{k=0}^{\infty}\frac{1}{k+1}
\left\langle \frac{d\frak m^t_{k}}{dt}(b(t),\ldots,b(t)),b(t)\right\rangle \\
&+\sum _{k=0}^{\infty}\frac{1}{k+1}
\left\langle \frak m^t_{k}(b(t),\ldots,\frac{db(t)}{dt},\ldots,b(t)),b(t)\right\rangle \\
&+\sum _{k=0}^{\infty}\frac{1}{k+1}\left\langle \frak m^t_{k}(b(t),\ldots,b(t)),\frac{db(t)}{dt}\right\rangle
\endaligned
\end{equation}
The sum of 2nd and the 3rd terms of (\ref{keisan1kou}) is:
$$
\sum _{k=0}^{\infty}\left\langle \frak m^t_{k}(b(t),\ldots,b(t)),\frac{db(t)}{dt}\right\rangle
= 0
$$
by cyclic symmetry and Maurer-Cartan equation of $b(t)$.
\par
We calculate the 1st term by using (\ref{isotopymaineq}) and obtain:
\begin{equation}\label{mondainotasizan}
\aligned
&-\sum_{k=0}^{\infty}\sum_{k_1+k_2=k+1}\sum_{i=0}^{k_1-1}
\frac{1}{k+1}\langle \frak c^t_{k_1}(\underbrace{b(t),\ldots,b(t)}_{i},
\frak m^t_{k_2}(b(t),\ldots),\ldots),b(t)
\\
&+\sum_{k=0}^{\infty}\sum_{k_1+k_2=k+1}\sum_{i=0}^{k_2-1}
\frac{1}{k+1}\langle \frak m^t_{k_2}(\underbrace{b(t),\ldots,b(t)}_{i},
\frak c^t_{k_1}(b(t),\ldots),\ldots),b(t)
\rangle
\endaligned
\end{equation}
We have
$$
\langle \frak c_{k_1}(\ldots
\frak m^t_{k_2}(b(t),\ldots),\ldots),b(t)
\rangle
=\langle \frak c^t_{k_1}(b(t),\ldots),
\frak m^t_{k_2}(b(t),\ldots)
\rangle 
$$
and
$$\aligned
\langle \frak m_{k_2}(\ldots
\frak c^t_{k_1}(b(t),\ldots),\ldots),b(t)
\rangle
&= \langle \frak m^t_{k_2}(b(t),\ldots),
\frak c^t_{k_1}(b(t),\ldots)
\rangle 
\\
&=-\langle \frak c^t_{k_2}(b(t),\ldots),
\frak m^t_{k_1}(b(t),\ldots)
\rangle 
\endaligned$$
by cyclic symmetry and \cite{F1} (56).
Therefore (\ref{mondainotasizan}) is equal to
\begin{equation}\label{kekkachikasi}
-\sum_{k=0}^{\infty}\sum_{k_1+k_2=k+1}
\langle \frak c^t_{k_2}(b(t),\ldots),
\frak m^t_{k_1}(b(t),\ldots)\rangle.
\end{equation}
Using Maurer-Cartan equation for $b(t)$ we find that 
(\ref{kekkachikasi}) is equal to
$$
\langle \frak c^t(0),\frak m^t(0)\rangle.
$$
By (\ref{dervativemt}) this cancels with the derivative of the 2nd term of 
(\ref{deff(t)}). Namely $f(t)$ is independent of $t$.
\end{proof}
\begin{dfn}
Let $(C,\langle \cdot \rangle,\{\frak m_{k,\beta}\},\{\frak m_{-1,\beta}\})$ be 
an inhomogeneous cyclic filtered $A_{\infty}$ algebra.
We call the function $\Psi : \mathcal M(C;\Lambda_+) \to \Lambda_+$,  
defined by
$$
\Psi(b) = \Psi'(b) + \sum_{\beta}T^{E(\beta)}\frak m^0_{-1,\beta}
$$
its
{\it superpotential}.
\end{dfn}
\section{Geometric realization}
\label{geomsec}
Let $M$ be a $3 \times 2$ dimensional symplectic manifold with 
$c^1(M) = 0$ and $L$ its relatively spin Lagrangian 
submanifold with vanishing Maslov index.
\par
In \cite{F1} Theorem 1.1, we defined a $G$-gapped cyclic filtered $A_{\infty}$ 
algebra $(\Lambda(L),\langle \cdot \rangle,\{\frak m_{k,\beta}^J\})$ 
on its de Rham complex. We also proved that its psedo-isotpy type 
is independent of the choice of $J$, perturbation etc.
The main result of this section is as follows.
\begin{thm}\label{inhomoexistmain}
If $J$ satisfies Assumption \ref{assumpJ}, then there exists 
$\frak m_{-1,\beta}^J \in \R$ such that 
$(\Lambda(L),\langle \cdot \rangle,\{\frak m_{k,\beta}^J\},\{\frak m_{-1,\beta}^J \})$
is an inhomogeneous cyclic and gapped filetered $A_{\infty}$ algebra.
\par
Moreover the pseudo-isotopy type of it depends only on $M,L,J$ and 
is independent on other choices involved in the definition.
\end{thm}
\begin{proof}
For $\beta \in H_2(M,L;\Z)$ let $\mathcal M_k(\beta;J)$ be the 
moduli space of stable $J$ holomorphic maps $v : (\Sigma,\partial \Sigma) \to 
(M,L)$ from bordered Riamann surface $\Sigma$ 
of genus 0 with connected nonempty boundary $\partial \Sigma$, and with 
$k$ boundary marked points, such that 
$v$ is of homology class $\beta$.
Let $ev = (ev_0,\ldots,ev_{k-1}) : \mathcal M_k(\beta;J) \to L^k$ be the 
evaluation maps at the boundary marked points.
(See \cite{FOOO080} Subsection 2.1.1.)
\par
In \cite{F1} Theorem 3.1 and Corollary 3.1, we proved an 
existence of its Kuranishi structure with the following properties:
\par\medskip
\begin{enumerate}
\item It is compatible with the forgetful map
\begin{equation}\label{forgetmap}
\mathfrak{forget}_{k,0} 
: \mathcal M_k(\beta;J) \to \mathcal M_0(\beta;J).
\end{equation}
(See \cite{F1} Section 3 for the precise definition of 
this compatibility.)
\item
For $k\ge 1$ the evaluation map 
$ev_0 : \mathcal M_k(\beta;J) \to L$ is weakly submersive, 
in the sense of \cite{FOOO080} Definition A1.13.
\item It is invariant under the cyclic permutation of the 
boundary marked points.
\item We consider the decomposition of the boundary:
\begin{equation}\label{bdcompati0}
\aligned
\partial \mathcal M_{k+1}(\beta)
=
\bigcup_{1\le i\le j+1 \le k+1} 
&\bigcup_{\beta_1+\beta_2=\beta}
 \\
&\mathcal M_{j-i+1}(\beta_1) {}_{ev_0} \times_{ev_i} 
\mathcal M_{k-j+i}(\beta_2).
\endaligned
\end{equation}
(See {\rm \cite{FOOO080}} Subsection {\rm 7.1.1.}) Then the restriction of the Kuranishi structure of 
$\mathcal M_{k+1}(\beta)$ to the left hand side 
coincides with the fiber product Kuranishi structure in 
the right hand side.
\item
We consider the decomposition 
\begin{equation}\label{partM0betadecomp}
\partial\mathcal M_{0}(\beta) = \bigcup_{\beta_1+\beta_2=\beta}
\left(\mathcal M_{1}(\beta_1)\, {}_{ev_0} \times_{ev_0} \mathcal M_{1}(\beta_2)
\right)/\Z_2.
\end{equation}
Then,
the fiber 
product Kuranishi structure on 
$\mathcal M_{1}(\beta_1)\, {}_{ev_0} \times_{ev_0} \mathcal M_{1}(\beta_2)$
(which is well-defined by $2$) coincides with the pull back of the Kuranishi structure to
$\partial\mathcal M_{0}(\beta)$.
\end{enumerate}
\par\smallskip
We remark that in general the decomposition of the boundary of $\partial\mathcal M_{0}(\beta)$
is given by
\begin{equation}\label{genedecompdsi0}
\aligned
\partial\mathcal M_{0}(\beta) = &\bigcup_{\beta_1+\beta_2=\beta}
\left(\mathcal M_{1}(\beta_1)\, {}_{ev_0} \times_{ev_0} \mathcal M_{1}(\beta_2)\right)/\Z_2 \\
&\quad\cup \bigcup_{\tilde \beta} \mathcal M^{\text{\rm cl}}_{1}(\tilde{\beta}) \,{}_{ev^{\text{\rm int}}_0}\times_M L.
\endaligned\end{equation}
Here $\mathcal M^{\text{\rm cl}}_{1}(\tilde{\beta})$ is the moduli space of stable 
maps of genus zero without boudary, one marked point and of homology class $\tilde{\beta} \in H_2(M;\Z)$.
The sum is taken over all $\tilde{\beta} \in H_2(M;\Z)$ which goes to 
$\beta$ by $i_* : H_2(M;\Z) \to H_2(M;L;\Z)$.
By Assumption \ref{assumpJ} the 2nd term of the right hand side of 
(\ref{genedecompdsi0}) is an empty set.
\par
Let $E_0>0$. Then in \cite{F1} Theorem 5.1 and Corollary 5.1, we proved the 
existence of system of continuous families of multisections 
on the above Kuranishi spaces $\mathcal M_k(\beta;J)$ with $\beta\cap[\omega] < E_0$
with the following properties: 
\par\medskip
\begin{enumerate}
\item The families of multisections are transversal to $0$.
\item It is compatible with the forgetful map (\ref{forgetmap}).
(See \cite{F1} Section 5 for the precise definition of this compatibility.)
\item For $k \ge 1$ the evaluation map $ev_0$ induces a submersion of its zero 
set, in the sense of \cite{F1} Definition 4.1.4. 
\item It is invariant under the cyclic permutation of the 
boundary marked points.
\item It is compatible with the identification (\ref{bdcompati0}).
\item It is compatible with the identification (\ref{partM0betadecomp}).
\end{enumerate}
\par\smallskip
Let $\rho_i \in \Lambda(L)$ ($i=1,\ldots,k$) be the differential forms on $L$.
In \cite{F1} Section 6 we defined
\begin{equation}\label{mkdefinition}
\frak m_{k,\beta}^{J,\frak s}(\rho_1,\ldots,\rho_k)
= \text{\rm Corr}(\mathcal M_{k+1}(\beta;J);((ev_1,\ldots,ev_k),ev_0))
(\rho_1\times\ldots\times\rho_k).
\end{equation}
Here the right hand side is the smooth correspondence associated to the 
above continuous family of perturbations. (See \cite{F1} Section 4.)
(Note that (\ref{mkdefinition}) depends on the choice of family of multisections.
The symbol $\frak s$ is put to clarify this dependence.)
\par
We next define $\frak m_{-1,\beta}^{J,\frak s}$.
Let $\text{\rm pt}$ be the space consisting of one point.
We have an obvious map $\text{\rm tri} : \mathcal M_{0}(\beta;J) \to \text{\rm pt}$.
Note $\Lambda(\text{\rm pt}) = \R$. Moreover
$$
\dim \mathcal M_{0}(\beta;J)  = \dim L -3 +\mu(\beta) =0.
$$
Therefore we have an $\R$ linear map:
$$
\text{\rm Corr}(\mathcal M_{0}^J(\beta);(\text{\rm tri},\text{\rm tri})) : \R \to \R.
$$
\begin{dfn}
For $\beta\cap [\omega] < E_0$, we put
$$
\frak m_{-1,\beta}^{J,\frak s}
= \text{\rm Corr}(\mathcal M_{0}^J(\beta);(\text{\rm tri},\text{\rm tri}))(1) \in \R.
$$
\end{dfn}
\begin{dfn}
\begin{enumerate}
\item An {\it inhomogeneous cyclic filtered $A_{\infty}$ algebra modulo $T^{E_0}$} 
is $(C,\langle \cdot \rangle,\{\frak m_{k,\beta}\mid  E(\beta) < E_0\},\{\frak m_{-1,\beta}^J \mid  E(\beta) < E_0\})$
such that $(C,\langle \cdot \rangle,\{\frak m_{k,\beta}\mid  E(\beta) < E_0\})$ is 
a cyclic filtered $A_{\infty}$ algebra modulo $T^{E_0}$ and $\frak m_{-1,\beta}^J \in \R$.
\item A {\it pseudo-isotopy of inhomogeneous cyclic filtered $A_{\infty}$ algebra modulo $T^{E_0}$} 
is 
$(C,\langle \cdot \rangle,\{\frak m_{k,\beta}\mid  E(\beta) < E_0\},\{\frak c_{k,\beta}\mid  E(\beta) < E_0\},\{\frak m_{-1,\beta} \mid  E(\beta) < E_0\})$
,  if $(C,\langle \cdot \rangle,\{\frak m_{k,\beta}\mid  E(\beta) < E_0\},\{\frak c_{k,\beta}\mid  E(\beta) < E_0\})$
is a pseudo-isotopy of cyclic filtered $A_{\infty}$ algebra modulo $T^{E_0}$
 (namely (\ref{inhomomainformula})
(\ref{isotopymaineq}) hold for $E(\beta) < E_0$) and   (\ref{dervativemt}) holds for $E(\beta) < E_0$. 
\end{enumerate}
\end{dfn}
The modulo $T^{E_0}$ version of Proposition \ref{gginvariance} and Theorem \ref{pseudoisowelldef}
can be proved by the same proof.
\par
$(\Lambda(L),\langle \cdot \rangle,\{\frak m_{k,\beta}^{J,\frak s}\},\{\frak m_{-1,\beta}^{J,\frak s} \})$ 
which we defined above 
is an  inhomogeneous cyclic filtered $A_{\infty}$ algebra modulo $T^{E_0}$.
\begin{prp}\label{pseudohomoinvmodT}
$(\Lambda(L),\langle \cdot \rangle,\{\frak m_{k,\beta}^{J,\frak s}\},\{\frak m_{-1,\beta}^{J,\frak s} \})$ 
is independent of the choice of Kuranishi structure and family of multisections $\frak s$ satisfying the 
properties listed in this section, up to pseudo-isotopy of inhomogeneous cyclic filtered $A_{\infty}$ algebra modulo $T^{E_0}$.
\end{prp}
\begin{proof}
Let us take two different choices of system of Kuranishi structures and of families of multisections.
We consider 
$
[0,1]\times \mathcal M_k(\beta;J)
$
and evaluation maps
$$
ev=(ev_0,\ldots,ev_{k-1}) : [0,1]\times \mathcal M_k(\beta;J) \to L^k, 
\quad ev_t :  [0,1]\times \mathcal M_k(\beta;J) \to [0,1].
$$
As in \cite{F1} Section 11 Lemmas 11.1, 11.2, we have a system of Kuranishi structures 
and continuous families of multisections on 
$[0,1]\times \mathcal M_k(\beta;J)$ with the following properties:
\par\medskip
\begin{enumerate}
\item The families of multisections are transversal to $0$.
\item It is compatible with the forgetful map $[0,1]\times$ (\ref{forgetmap}).
\item For $k \ge 1$ the evaluation map 
$$
(ev_t,ev_0) : [0,1]\times \mathcal M_k(\beta;J) \to [0,1] \times L
$$
is weakly submersive and
 induces a submersion of the zero 
set of family of multisections, in the sense of \cite{F1} Definition 4.1.4. 
\item They are invariant under the cyclic permutation of the 
boundary marked points.
\item It is compatible with the identification (\ref{bdcompati0}).
\item It is compatible with the identification (\ref{partM0betadecomp}).
\item 
$$
ev_t : [0,1]\times\mathcal M_0(\beta) \to [0,1]
$$
is  weakly submersive and
 induces a submersion on the zero 
set of family of multisections, in the sense of \cite{F1} Definition 4.1.4. 
\item 
At $t_0=0,1$ the induced Kuranishi structure and families of multisecitons on 
$\{t_0\}\times\mathcal M_k(\beta)$ coincides with  
given two choices of Kuranishi structures and of families of multisections.
\end{enumerate}
\par\smallskip
In \cite{F1} Section 11, we defined a pseudo-isotopy of cyclic filtered 
$A_{\infty}$ algebra as follows. Let $\rho_1,\ldots,\rho_k 
\in \Lambda(L)$. 
We put 
\begin{equation}\label{46}
\aligned
&\text{\rm Corr}_*([0,1]\times\mathcal M_{k+1}(\beta;J);(ev_{1},\ldots,ev_k),(ev_t,ev_0))
(\rho_1\times\ldots\times\rho_k) \\
&=
\rho(t) + dt \wedge \sigma(t),
\endaligned\end{equation}
and define
\begin{equation}\label{pseudoisoshiki}
\frak m^t_{k,\beta}(\rho_1,\ldots,\rho_k) = \rho(t),  \qquad
\frak c^t_{k,\beta}(\rho_1,\ldots,\rho_k) = \sigma(t).
\end{equation}
\par
We next define $\frak m_{-1,\beta}^t$ .
Let $\text{\rm tri} : [0,1]\times\mathcal M_{0}(\beta;J) \to \text{\rm pt}$
be an obvious map to a point. We take $1 \in \Lambda^0(\text{\rm pt}) = \R$ and
put
\begin{equation}
\text{\rm Corr}_*([0,1]\times\mathcal M_{0}(\beta;J);\text{\rm tri},ev_t)(1)
=
\rho(t) + dt \wedge \sigma(t).
\end{equation}
We then define 
\begin{equation}
\frak m^t_{-1,\beta} = \rho(t).
\end{equation}
\begin{lmm}
$(\Lambda(L),\langle \cdot \rangle,\{\frak m_{k,\beta}^{t}\},\{\frak c_{k,\beta}^{t}\},
\{\frak m_{-1,\beta}^{t} \})$  above defines a 
pseudo-isotopy of inhomogeneous cyclic filtered $A_{\infty}$
algebra modulo $T^{E_0}$.
\end{lmm}
\begin{proof}
In \cite{F1} Section 11 it is proved that 
$(\Lambda(L),\langle \cdot \rangle,\{\frak m_{k,\beta}^{t}\},\{\frak c_{k,\beta}^{t}\})$
is a pseudo-isotopy of cyclic filtered $A_{\infty}$
algebra modulo $T^{E_0}$.  Therefore it suffices to check (\ref{dervativemt}).
\par
Let $0 \le t_1 < t_2 \le 1$.  We have:
$$\aligned
&\partial \left(
 [t_1,t_2]\times\mathcal M_{0}(\beta;J)
\right) \\
&= \left(
\{t_1,t_2\}\times\mathcal M_{0}(\beta;J) 
\right)\\
&\cup 
\bigcup_{\beta_1+\beta_2 = \beta}
\left(
 ([t_1,t_2]\times \mathcal M_1(\beta_1))\, {}_{(ev_0,ev_t)}\times_{(ev_0,ev_t)}\,
  ([t_1,t_2]\times \mathcal M_1(\beta_2))
\right)/\Z_2.
\endaligned$$
We now apply Stokes' formula (\cite{F1} Proposition 4.2) to 
the closed 1 form $ev^*_t(dt)$ on the zero set of multisections on 
$[t_1,t_2]\times\mathcal M_{0}(\beta;J)$
and obtain:
$$
\frak m_{-1,\beta}^{t_2} - \frak m_{-1,\beta}^{t_1}
= \sum_{\beta_1+\beta_2 = \beta}
\int_{t_1}^{t_2} \langle \frak c_{0,\beta_1}^t(1), 
\frak m_{0,\beta_2}^t(1) \rangle dt
$$
By taking $t_2$ derivative we obtain (\ref{dervativemt}).
\end{proof}
The proof of Proposition \ref{pseudohomoinvmodT} is now complete.
\end{proof}
We thus proved mod $T^{E_0}$ version  of Theorem \ref{inhomoexistmain}.
We next prove the following inhomogeneous version of Theorem 8.1 \cite{F1}.
\begin{lmm}\label{pisoextlem}
Let $0 < E_0 < E_1$ and 
$(C,\langle\cdot\rangle,\{\frak m_{k,\beta}^{i}\},\{\frak m_{-1,\beta}^{i}\})$
be $G$-gapped inhomogeneous cyclic filtered $A_{\infty}$ algebra modulo $T^{E_i}$, 
for $i=0,1$.
Let $(C,\langle\cdot\rangle,\{\frak m_{k,\beta}^{t}\},\{\frak c_{k,\beta}^{t}\},\{\frak m_{-1,\beta}^{t}\})$ be a pseudo-isotpy of  $G$-gapped inhomogeneous cyclic filtered $A_{\infty}$ algebra modulo $T^{E_0}$ between them.
\par
Then, $(C,\langle\cdot\rangle,\{\frak m_{k,\beta}^{i}\},\{\frak m_{-1,\beta}^{i}\})$ can be extended 
to a $G$-gapped inhomogeneous cyclic filtered $A_{\infty}$ algebra modulo $T^{E_1}$ and 
$(C,\langle\cdot\rangle,\{\frak m_{k,\beta}^{t}\},\{\frak c_{k,\beta}^{t}\},\{\frak m_{-1,\beta}^{t}\})$ 
can be extended to a pseudo-isotpy of  $G$-gapped inhomogeneous cyclic filtered $A_{\infty}$ algebra modulo $T^{E_1}$ between them.
\end{lmm}
\begin{proof}
We may assume that $G \cap [E_0,E_1) = \{E_0\}$. In 
\cite{F1} Theorem 8.1 the extension to cyclic filtered $A_{\infty}$ algebra mod $T^{E_1}$ and 
extension to pseudo-isotopy of cyclic filtered $A_{\infty}$ algebra mod $T^{E_1}$ are obtained.
So it suffices to find $\frak m_{-1,\beta}^t$ for 
$E(\beta) = E_0$.
We define 
$$
\frak m_{-1,\beta}^t  = \frak m_{-1,\beta}^1 
+ \sum_{\beta_1+\beta_2=\beta} \int_t^1 
\langle \frak c_{0,\beta_1}^t(1), 
\frak m_{0,\beta_2}^t(1) \rangle dt.
$$
It is easy to check (\ref{dervativemt}).
\end{proof}
We next construct gapped inhomogeneous cyclic filtered $A_{\infty}$ algebra 
\linebreak
$(\Lambda(L),\langle\cdot\rangle,\{\frak m_{k,\beta}\},
\{\frak m_{-1,\beta}\})$.
Let $E_i$ be sequence $0 < \ldots < E_i < E_{i+1} < \ldots$. 
We obtain a sequence 
$(\Lambda(L),\langle\cdot\rangle,\{\frak m^i_{k,\beta}\},\{\frak m^i_{-1,\beta}\})$
of inhomogenuous cyclic filtered $A_{\infty}$ algebra modulo $T^{E_i}$ for each $i$.
By Proposition \ref{pseudohomoinvmodT} we have 
a pseudo-isotopy of inhomogenuous cyclic filtered $A_{\infty}$ algebra modulo $T^{E_i}$
$(\Lambda(L),\langle\cdot\rangle,\{\frak m^{i,t}_{k,\beta}\},\{\frak c^{i,t}_{k,\beta}\},
\{\frak m^{i,t}_{-1,\beta}\})$
between $(\Lambda(L),\langle\cdot\rangle,\{\frak m^i_{k,\beta}\},\{\frak m^i_{-1,\beta}\})$ 
and $(\Lambda(L),\langle\cdot\rangle,\{\frak m^{i+1}_{k,\beta}\},\{\frak m^{i+1}_{-1,\beta}\})$.
\par
We then can use Lemma \ref{pisoextlem} in the same way as \cite{F1} Section 12 
and \cite{FOOO080} Section 7.2, to extend 
$(\Lambda(L),\langle\cdot\rangle,\{\frak m^i_{k,\beta}\},\{\frak m^i_{-1,\beta}\})$ to 
an inhomogenuous cyclic filtered $A_{\infty}$ algebra 
and $(\Lambda(L),\langle\cdot\rangle,\{\frak m^{i,t}_{k,\beta}\},\{\frak c^{i,t}_{k,\beta}\},
\{\frak m^{i,t}_{-1,\beta}\})$ 
to a pseudo-isotopy of  inhomogenuous cyclic filtered $A_{\infty}$ algebra between them.
They are isomorphic to each other. Therefore we have $(\Lambda(L),\langle\cdot\rangle,\{\frak m_{k,\beta}\},
\{\frak m_{-1,\beta}\})$.
\par
We can prove that it is independent of the choice of system of Kuranishi structures and continuous families of 
multisections in the same way as \cite{F1} Section 14 by working out the inhomogeneous version of 
pseudo-isotpy of pseudo-isotopies. We omit the detail of it. 
Instead, we complete the proof of Theorem \ref{superpotential} directly
without using inhomogeneous version of pseudo-isotpy of pseudo-isotopies but 
uses only the result of  \cite{F1} Section 14 and ones of this paper. 
\par
Let 
$(\Lambda(L),\langle\cdot\rangle,\{\frak m^{i \prime}_{k,\beta}\},\{\frak m^{i \prime}_{-1,\beta}\})$
be an inhomogeneour cyclic filtered $A_{\infty}$ algebras modulo $T^{E_i}$ obtained by alternative 
choices and
$(\Lambda(L),\langle\cdot\rangle,\{\frak m^{i,t \prime}_{k,\beta}\},\{\frak c^{i,t \prime}_{k,\beta}\},
\{\frak m^{i,t \prime}_{-1,\beta}\})$ a pseudo-isotopies modulo $T^{E_i}$ 
of inhomogeneour cyclic filtered $A_{\infty}$ algebras. We first extend them to 
inhomogeneour cyclic filtered $A_{\infty}$ algebras and pseudo-isotopies among them.
\par
By Proposition \ref{pseudohomoinvmodT}, 
$(\Lambda(L),\langle\cdot\rangle,\{\frak m^{i \prime}_{k,\beta}\},\{\frak m^{i \prime}_{-1,\beta}\})$ is 
pseudo-isotopic modulo $T^{E_i}$ to 
$(\Lambda(L),\langle\cdot\rangle,\{\frak m^{i \prime}_{k,\beta}\},\{\frak m^{i \prime}_{-1,\beta}\})$.
By \cite{F1} Theorem 14.1, this pseudo-isotpy modulo $T^{E_i}$ extends to  
a pseudo-isotpy of cyclic $A_{\infty}$ algebra. 
(We do not use the fact that it extends to pseudo-isotpy of {\it inhomogeneous} 
cyclic $A_{\infty}$ algebras here.)
Therefore by moduo $T^{E_i}$ version of Theorem \ref{pseudoisowelldef}, 
we have an isomorphism
$$
(\frak f_i)_* : \mathcal M(\Lambda(L),\{\frak m^{i}_{k,\beta}\};\Lambda_+) \cong 
\mathcal M(\Lambda(L),\{\frak m^{i \prime}_{k,\beta}\};\Lambda_+).
$$
By modulo $T^{E_i}$ version of Theorem \ref{inhomoexistmain} we have
\begin{equation}\label{410}
\Psi((\frak f_i)_* (b)) \equiv \Psi(b) \mod T^{E_i}.
\end{equation}
For $i>j$, let
$$
(\frak c_{i,j})_* : \mathcal M(\Lambda(L),\{\frak m^{j}_{k,\beta}\};\Lambda_+) \cong 
\mathcal M(\Lambda(L),\{\frak m^{i}_{k,\beta}\};\Lambda_+).
$$
and
$$
(\frak c ^{\prime}_{i,j})_* : \mathcal M(\Lambda(L),\{\frak m^{j \prime}_{k,\beta}\};\Lambda_+) \cong 
\mathcal M(\Lambda(L),\{\frak m^{i \prime}_{k,\beta}\};\Lambda_+).
$$
be the isomorphisms induced by the pseudo-isotopies. We have
\begin{equation}\label{411}
\Psi((\frak c ^{\prime}_{i,j})_*(b)) = \Psi((\frak c ^{\prime}_{i,j})_*(b)).
\end{equation}
Furthermore the construction of pseudo-isotopy of pseudo-isotopies in 
\cite{F1} Section 14 imply
\begin{equation}\label{412}
(\frak f_j)_*\circ (\frak c ^{\prime}_{i,j})_* = (\frak c ^{\prime}_{i,j})_*\circ (\frak f_i)_*.
\end{equation}
(\ref{410}), (\ref{411}), (\ref{412}) immediately imply
$$
\Psi((\frak f_1)_* (b)) = \Psi(b).
$$
We thus proved Theorem \ref{superpotential}.3.
The proof of Theorem \ref{superpotential} is now complete.
\end{proof}
\section{Relation to canonical model}
\label{canonicalmodelsec}
In \cite{FOOO080} Subsection 5.4.4 and \cite{F1} Section 10, we defined 
canonical model
$(H,\langle \cdot \rangle,\{ \frak m_{k,\beta}^{\text{\rm can}}\})$ 
of $G$-gapped cyclic filtered $A_{\infty}$ algebra 
$(C,\langle \cdot \rangle,\{ \frak m_{k,\beta}\})$.
(We assumed $\overline C$ is either finite dimensional or 
de Rham complex $\Lambda(L)$.)
We also constructed a 
$G$-gapped cyclic filtered $A_{\infty}$ homomorphism 
$\frak f : H \to C$, which is a homotopy equivalence.
Suppose that 
$(C,\langle \cdot \rangle,\{ \frak m_{k,\beta}\},\{\frak m_{-1,\beta}
\})$ is an {\it inhomegeneous} $G$-gapped cyclic filtered $A_{\infty}$ algebra.
In this section, we will define 
$\frak m_{-1,\beta}^{\text{\rm can}}$ so that
$\frak f_* : \mathcal M(H;\Lambda_+) \to  \mathcal M(C;\Lambda_+)$
preserves superpotential.
\par
To define $\frak m_{-1,\beta}^{\text{\rm can}}$
we need some notations. We use results and notations of 
\cite{F1} Sections 9 and 10 in this section.
\par
Let $T$ be a ribbon tree. Let $C_0(T)$ be the set of vertices. 
We assume that we have its decomposition 
$C_0(T) = C^{\text{\rm int}}_0(T) \sqcup C^{\text{\rm ext}}_0(T)$ 
to interior vertices and exterior vertices.
Let $\beta(\cdot) : C^{\text{\rm int}}_0(T) \to G$ be a map to a 
discrete submonoid $G$ of $\R_{\ge 0}$.
\begin{dfn}
We denote by $Gr^-(k,\beta)$ the set of $\Gamma =(T,C^{\text{\rm int}}_0(T),C^{\text{\rm ext}}_0(T),\beta(\cdot))$
such that:
(1)  $\sum_{v \in C^{\text{\rm int}}_0(T)}\beta(v) = \beta$.
(2) $\#C^{\text{\rm ext}}_0(T) = k$. 
(3)  If $\beta(v) = 0$, then $v$ has at least 3 edges.
\par
The automorphism group $\text{\rm Aut}(\Gamma)$ of an element 
$\Gamma =(T,C^{\text{\rm int}}_0(T),C^{\text{\rm ext}}_0(T),\beta(\cdot))$ of $Gr^-(k,\beta)$ 
is the set of isomorphisms $\phi : T \to T$  of ribbon tree which preserves the 
decomposition $C_0(T) = C^{\text{\rm int}}_0(T) \sqcup C^{\text{\rm ext}}_0(T)$ 
and such that $\beta(\phi(v)) = \beta(v)$. 
\end{dfn}
We remark that 
$k=0,1,\ldots$ in  $Gr^-(k,\beta)$. The case $k=0$ is included.
We also remark that the automorphism of {\it rooted} ribbon tree is trivial.
\par
Let $(v,e)$ be a flag of $\Gamma$, that is a pair of an interior 
vertex $v$ and an edge $e$ containing $v$.
Let $b \in \overline C^1$.
We are going to define
$\frak m(\Gamma;b) \in \R$.
\par
Let $T_0,\ldots,T_{\ell}$ be the irreducible components of 
$\Gamma \setminus v$. We enumerate them so that
$e \in T_0$ and they respect counter clockwise cyclic order of $\R^2$. 
Together with the data induced from $\Gamma$, the tree $T_i$ defines an element 
$\Gamma_i \in Gr(k_i,\beta_i)$. Here $Gr(k_i,\beta_i)$ is as in 
\cite{F1} Definition 9.1. Namely its element is an element of 
$Gr^-(k_i,\beta_i)$ together with a choice of a base point which 
is an exterior vertex. In our situation the base points of $\Gamma_i$ are $v$ for all $i$.
\begin{dfn}\label{mgamma}
$$
\frak m(\Gamma,v,e;b) 
= \langle \frak m_{\ell,\beta(v)}(\frak f_{\Gamma_1}(b,\ldots,b),\ldots,
\frak f_{\Gamma_{\ell}}(b,\ldots,b)),\frak f_{\Gamma_0}(b,\ldots,b)\rangle.
$$
Here $\frak f_{\Gamma}$ is defined in \cite{F1} section 10.
\end{dfn}
We remark that there is no sign in Definition \ref{mgamma},
since the degree of $b$ after shifted is even.
\begin{lmm}
$\frak m(\Gamma,v,e;b)$ is independent of $v$ and $e$ and depends only 
on $\Gamma$ and $b$.
\end{lmm}
This is Proposition 10.1 \cite{F1}.
Hereafter we write $\frak m(\Gamma;b)$ in place of $\frak m(\Gamma,v,e;b)$.
\begin{dfn}\label{m-1can}
$$
\frak m^{\text{\rm can}}_{-1,\beta} = \sum_{\Gamma \in Gr^-(0,\beta)} 
\frac{\frak m(\Gamma)}{\#\text{\rm Aut}(\Gamma)}.
$$
\end{dfn}
We remark that we write $\frak m(\Gamma)$ instead of $\frak m(\Gamma;b)$, 
since in the case of $\Gamma \in Gr^-(0,\beta)$ there is no exterior vertex and 
hence $b$ never appears.
\par
$(H,\langle \cdot \rangle,\{ \frak m_{k,\beta}^{\text{\rm can}}\},
\{ \frak m_{-1,\beta}^{\text{\rm can}}\})$ is an 
inhomegeneous $G$-gapped cyclic filtered $A_{\infty}$ algebra.
Let
$$
\Psi^{\text{\rm can}} : \mathcal M(H;\Lambda_+) \to \Lambda_+
$$
be its superpotential.
The filtered $A_{\infty}$ homomorphism $\frak f : H \to C$ 
induces 
$
\frak f_* : \mathcal M(H;\Lambda_+) \to \mathcal M(C;\Lambda_+)
$
by
\begin{equation}
\frak f_*(b) = \sum_{k=0}^{\infty}\sum_{\beta\in G}
T^{E(\beta)}\frak f_{k,\beta}(b,\ldots,b).
\end{equation}
The main result of this section is:
\begin{thm}\label{canmainth}
\begin{equation}
\Psi(\frak f_*(b)) = \Psi^{\text{\rm can}}(b).
\end{equation}
\end{thm}
\begin{rem}
We consider the case of $\overline C = \Lambda(L)$ with 
$H^1(L;\R) = 0$.
Then since $H^1 = 0$, the set $\mathcal M(H;\Lambda_+)$ consists 
of one point $0$. Therefore $\mathcal M(C;\Lambda_+)$  also consists of one point.
The invariant of Corollary \ref{maincor} is the value of superpotential at this point.
\par
Theorem \ref{canmainth} implies that this invariant is 
\begin{equation}\label{sumtreeformula}
\sum_{\beta\in G}  T^{E(\beta)} \frak m_{-1,\beta}^{\text{\rm can}}
= 
\sum_{\beta\in G} \sum_{\Gamma \in Gr^{-1}(0,\beta)}
\frac{T^{E(\beta)}}{\#\text{\rm Aut}(\Gamma)}\frak m(\Gamma).
\end{equation}
\end{rem}
\begin{proof}[Proof of Theorem \ref{canmainth}]
Let $b\in H^1_+ = H^1 \otimes_{\R} \Lambda_+$. We define
\begin{equation}
\Phi(b) = \sum_{k=0}^{\infty}\sum_{\beta \in G}
\sum_{\Gamma \in Gr^{-1}(k,\beta)}
\frac{T^{E(\beta)}}{\#\text{\rm Aut}(\Gamma)} \frak m(\Gamma;b) \in \Lambda_+.
\end{equation}
\begin{lmm}\label{PhiisPsican}
$$
\Phi(b) = \Psi^{\text{\rm can}}(b).
$$
\end{lmm}
\begin{proof}
In view of Definition \ref{m-1can} it suffices to prove:
\begin{equation}\label{canmainformula1}
\langle \frak m_{k,\beta}(b,\ldots,b),b\rangle
= (k+1) \sum_{\Gamma \in Gr^-(k+1,\beta)} 
\frac{\frak m(\Gamma;b)}{\#\text{\rm Aut}(\Gamma)}.
\end{equation}
We will prove (\ref{canmainformula1}) below.
\par
Let $\Gamma \in Gr^-(k+1,\beta)$.
Let $\{v_0,\ldots,v_k\} = C_0^{\text{ext}}(\Gamma)$
such that $v_0,\ldots,v_k$ respects the counter clockwise cyclic 
order of $\R^2$.
Let $e_i$ be the unique edge containing $v_i$.
We define $v'_i$ by $\partial e = \{v_i,v'_i\}$.
\par
By definition we have:
$$
\frak m_{k,\beta}(b,\ldots,b) = \sum_{\Gamma \in Gr^-(k+1,\beta)}\sum_{i=0}^k
\frac{\frak m_{(\Gamma,v_i)}(b,\ldots,b)}{\#\text{\rm Aut}(\Gamma)}.
$$
This is because $(\Gamma,v_i) \in Gr(k+1,\beta)$ and $(\Gamma,v_i)$ 
is the same element as $(\Gamma,v_j)$ in $ Gr(k+1,\beta)$ if and only if 
there exists an element of $\phi \in \text{\rm Aut}(\Gamma)$ such that 
$\phi(v_i) = v_j$.
\par
Moreover 
$$
\langle \frak m_{\Gamma,v_i}(b,\ldots,b),b\rangle
=\frak m(\Gamma,v_i,e_i;b,\ldots,b),
$$
where the right hand side in defined in Definition 10.1 \cite{F1}.
By Proposition 10.1 \cite{F1}, 
$\frak m(\Gamma,v_i,e_i;b,\ldots,b)$ is independent of $i$ and is
$\frak m(\Gamma;b)$.
This implies (\ref{canmainformula1}).
The proof of Lemma \ref{PhiisPsican} is complete.
\end{proof}
The next proposition completes the proof of 
Theorem  \ref{canmainth}.
\end{proof}
\begin{prp}\label{canmaincanprp}
If $b \in \widetilde{\mathcal M}(H;\Lambda_+)$, then we have:
\begin{equation}
\Phi(b) = \Psi^{\text{\rm can}}(\frak f_*(b)).
\end{equation}
\end{prp}
\begin{proof}
\begin{lmm}\label{mne01sum}
\begin{equation}\label{sumovermnon0}
\aligned
&\sum_{(\ell,\beta') \ne (1,0)} \frac{T^{E(\beta)}}{\ell+1} 
\langle \frak m_{\ell,\beta'}(\frak f_*(b),\ldots,\frak f_*(b)),\frak f_*(b)\rangle \\
&\quad= \sum_{k=0}^{\infty}\sum_{\beta\in G}\sum_{\Gamma \in Gr^-(k,\beta)}
\frac{T^{E(\beta)}}{\#\text{\rm Aut}(\Gamma)}\#C_0^{\text{\rm int}}(\Gamma) 
\frak m(\Gamma;b).
\endaligned
\end{equation}
\end{lmm}
\begin{proof}
Let $\Gamma \in Gr^-(k,\beta)$ and $(v,e)$ its flag. We 
obtain the irreducible components 
$\Gamma_0,\ldots,\Gamma_{\ell}$ of $\Gamma \setminus v$ as before. By definition we have
\begin{equation}\label{53tochuu}
\langle \frak m_{\ell,\beta(v)}(\frak f_{\Gamma_1}(b,\ldots,b),\ldots,\frak f_{\Gamma_{\ell}}(b,\ldots,b)),
\frak f_{\Gamma_0}(b,\ldots,b)\rangle
= \frak m(\Gamma;b).
\end{equation}
We remark that the right hand side is independent of $(v,e)$ by Proposition 10.1 \cite{F1}.
If we take the sum of (\ref{53tochuu}) over all $\Gamma,v$ with weight $T^{E(\beta)}/\#\text{\rm Aut}(\Gamma)$ 
then we obtain the right hand side of 
(\ref{sumovermnon0}). 
On the other hand, if we take the sum of the right hand side of (\ref{53tochuu}) over all $\Gamma,v,e$ with 
weight $T^{E(\beta)}$ we obtain
$$
\sum_{(\ell,\beta') \ne (1,0)} T^{E(\beta')}
\langle \frak m_{\ell,\beta'}(\frak f_*(b),\ldots,\frak f_*(b)),\frak f_*(b)\rangle.
$$
Since the choice of $e$ for given $\Gamma, v$ is $\ell+1$, 
we obtain (\ref{sumovermnon0}).
\end{proof}
\begin{lmm}\label{m01sum}
\begin{equation}\label{sumoverm0}
\aligned
&\langle \frak m_{1,0}(\frak f_*(b)),\frak f_*(b)\rangle \\
&= -2 \sum_{k=0}^{\infty}\sum_{\beta\in G}\sum_{\Gamma \in Gr^-(k,\beta)}
\frac{T^{E(\beta)}}{\#\text{\rm Aut}(\Gamma)}\#C_1^{\text{\rm int}}(\Gamma) 
\frak m(\Gamma;b).
\endaligned
\end{equation}
\end{lmm}
Here $C_1^{\text{\rm int}}(\Gamma)$ is the set of interior edges.
\begin{proof}
Let $(v,e)$ be a flag of $\Gamma \in Gr^-(k,\beta)$ such that $e$ is an interior edge.
We define $\frak m'(\Gamma,e,v;b)$ as follows.
Let $T_{(0)}$, $T'_{(1)}$ be the irreducible components of $\Gamma \setminus e$ such that 
$T_{(0)}$ contains $v$. We put $T_{(1)} = T'_{(1)} \cup e$.
Using the data induced from $\Gamma$, the trees $T_{(0)}$, $T_{(1)}$ induce 
$\Gamma_{(0)} \in Gr(k_{(0)},\beta_{(0)})$, $\Gamma_{(1)}\in Gr(k_{(1)},\beta_{(1)})$.
(The roots of $\Gamma_{(0)}$, $\Gamma_{(1)}$ are $v$.)
We define
\begin{equation}
\frak m'(\Gamma,e,v;b) 
= \langle 
\frak m_{1,0}(\frak f_{\Gamma_{(1)}}(b,\ldots,b)),\frak f_{\Gamma_{(0)}}(b,\ldots,b))
\rangle.
\end{equation}
Let $v_0,\ldots,v_k$ be the set of exterior vertices of $\Gamma$.
Let $e_i$ be the edge containing $v_i$ and and $\partial e_i =\{v_i,v'_i\}$.
\begin{sublmm}\label{sublemasum}
If $v \ne v'_i$, $(i=0,\ldots,k)$ then 
\begin{equation}
\frak m'(\Gamma,e,v;b)  =  -\frak m(\Gamma;b).
\end{equation}
If $v= v'_i$ then
\begin{equation}\label{sublemspec}
\frak m'(\Gamma,e,v;b)  =  -\frak m(\Gamma;b) 
+ \langle
\frak m_{(\Gamma,v_i)}(b,\ldots,b),b\rangle.
\end{equation}
Here $(\Gamma,v_i) \in Gr(k,\beta)$ and $\frak m_{(\Gamma,v_i)}$
is defined in {\rm \cite{F1}} Section {\rm 10}.
\end{sublmm}
\begin{proof}
We use Lemma 10.1 \cite{F1}, its proof and notations there, during the proof 
of Sublemma \ref{sublemasum}.
\par
Let $\Gamma,v,e$ be as in Sublemma \ref{sublemasum}. 
We put $\partial e = \{v,v'\}$.
Let $T_0,\ldots,T_m$ be the irreducible components of 
$\Gamma \setminus v'$. We enumerate them so that 
$v \in T_0$ and it respects counter clockwise cyclic order of $\R^2$.  
$T_i$ together with the data induced from $\Gamma$ becomes $\Gamma_i$, 
whose root is $v'$.
By definition 
$$
\Gamma_{(1)} = \Gamma_1 \cup \ldots \cup \Gamma_m \cup e.
$$
Therefore by the definition in \cite{F1} Section 10, we have
\begin{equation}
\frak f_{\Gamma_{(1)}}(b,\ldots,b) 
= (G\circ \frak m_{m,\beta(v')})(\frak f_{\Gamma_1}(b,\ldots,b),\ldots,\frak f_{\Gamma_m}(b,\ldots,b)).
\end{equation}
Therefore 
\begin{equation}\aligned
\frak m'(\Gamma,e,v;b) 
= \langle
(\frak m_{1,0}\circ G\circ \frak m_{m,\beta(v')})(&\frak f_{\Gamma_1}(b,\ldots,b),\ldots,\\
&\frak f_{\Gamma_m}(b,\ldots,b)),
\frak f_{\Gamma_{(0)}}(b,\ldots,b)
\rangle.
\endaligned
\end{equation}
By Lemma 10.1 \cite{F1} we have
\begin{equation}
\frak m_{1,0} \circ G = - G \circ \frak m_{1,0} + \Pi - \text{\rm identity}.
\end{equation}
\par
We first assume $v \ne v'_i$. Then $\Gamma_{(0)} \in Gr(k_{(0)},\beta_{(0)})$ with 
$(k_{(0)},\beta_{(0)}) \ne (1,0)$. It follows that
$\frak f_{\Gamma_{(0)}}(b,\ldots,b) \in \text{\rm Im}\, G$.
We remark that 
$$
\langle \text{\rm Im}\, G, \text{\rm Im}\, G + \text{\rm Im}\, \Pi\rangle = 0.
$$
Therefore 
$$
\aligned
\frak m'(\Gamma,e,v;b) &= 
-\langle
\frak m_{m,\beta(v')}(\frak f_{\Gamma_1}(b,\ldots,b),\ldots,
\frak f_{\Gamma_m}(b,\ldots,b)),
\frak f_{\Gamma_{(0)}}(b,\ldots,b)
\rangle \\
&=
-\frak m(\Gamma;b),
\endaligned
$$
as required. 
\par
If $v=v_i$ then $\frak f_{\Gamma_{(0)}}$ is identity. Therefore
$$
\aligned
\frak m'(\Gamma,e,v;b) &= 
-\langle
\frak m_{m,\beta(v')}(\frak f_{\Gamma_1}(b,\ldots,b),\ldots,
\frak f_{\Gamma_m}(b,\ldots,b)),
b
\rangle \\
&\quad 
+\langle
(\Pi \circ \frak m_{m,\beta(v')})(\frak f_{\Gamma_1}(b,\ldots,b),\ldots,
\frak f_{\Gamma_m}(b,\ldots,b)),
b
\rangle 
\\
&=
-\frak m(\Gamma;b)
+ \langle
\frak m_{(\Gamma,v_i)}(b,\ldots,b),b\rangle.
\endaligned
$$
The proof of sublemma is complete.
\end{proof}
Using Maurer-Cartan equation for $b$ we find
$$
\sum_{k=0}^{\infty}\sum_{\beta\in G}\sum_{\Gamma\in Gr^-(k,\beta)}
\sum_{i=0}^k
\frac{T^{E(\beta)}}{\#\text{\rm Aut}(\Gamma)}\frak m_{(\Gamma,v_i)}(b,\ldots,b) = 0.
$$
Therefore the sum of the second term of (\ref{sublemspec}) vanishes.
Lemma \ref{m01sum} now follows from Sublemma \ref{sublemasum}.
\end{proof}
Since $\Gamma$ is a tree we have
$\#C_0^{\text{\rm int}}(\Gamma) - \#C_1^{\text{\rm int}}(\Gamma) = 1$.
Therefore Lemmas \ref{mne01sum} and \ref{m01sum} imply Proposition 
\ref{canmaincanprp}.
\end{proof}
Using the proof of Theorem \ref{canmainth} and \cite{F1} Section 9, we can prove the 
following:
\begin{thm}\label{canpseudoisot}
If two gapped inhomogeneous cyclic filtered $A_{\infty}$ algebras 
are pseudo-isotopic  
to each other, then 
so are their canonical models.
\end{thm}
We omit the proof since it is a straightforward analog and we do not use Theorem \ref{canpseudoisot} 
in this paper.
\section{Wall crossing formula}
\label{wallcrosssec}
In this section we prove Theorem \ref{wallcross}.
We first review the definition of the number (\ref{wallcrossformula}) in more detail.
\par
We remark that (\ref{wallcrossformula}) is a rational number since we can use 
multi (but finitely many) valued section of $\mathcal M_1^{\text{\rm cl}}(\alpha;\mathcal J)$
to define it. (The argument to do so is the same as \cite{FO}.)
\par
On the other hand, to prove Theorem \ref{wallcross} we need to choose a perturbation 
of  $\mathcal M_1^{\text{\rm cl}}(\alpha;\mathcal J)$ so that 
it is compatible with one in 
$\mathcal M_k(\beta;\mathcal J)$. 
Here
\begin{equation}\label{famidiscmod}
\mathcal M_k(\beta;\mathcal J)
= \bigcup_{t \in [0,1]} \{t\} \times \mathcal M_k(\beta;J_t).
\end{equation}
Since we use {\it continuous family of} multi-sections to perturb 
$\mathcal M_k(\beta;\mathcal J)$, we need to use continuous family of multi-sections
also for $\mathcal M_1^{\text{\rm cl}}(\alpha;\mathcal J)$.
Actually this is the way taken in \cite{F1} Sections 3 and 5.
\par
There exists a Kuranishi structure and continuous family of multi-sections 
on $\mathcal M_1^{\text{\rm cl}}(\alpha;\mathcal J)$ with the following properties:
\par\medskip
\begin{enumerate}
\item The evaluation map
\begin{equation}\label{evaluationtint}
(ev_t,ev^{\text{\rm int}}) : \mathcal M_1^{\text{\rm cl}}(\alpha;\mathcal J) 
\to [0,1] \times M
\end{equation}
is weakly submersive. 
\item 
Continuous family of multi-sections is transversal to $0$ and 
(\ref{evaluationtint}) induces submersion on its zero set. 
\item
The image of the restriction of $(ev_t,ev^{\text{\rm int}})$ to the zero set of 
continuous family of multi-sections is disjoint from 
$\{0,1\} \times L$.
\end{enumerate}
\par\smallskip
This is Lemmas 3.2 and 5.3 of \cite{F1}.
Let $\text{\rm tri} : \mathcal M_1^{\text{\rm cl}}(\alpha;\mathcal J)
\to \text{\rm pt}$ be the trivial map. We use the above continuous family of multisections 
and define
\begin{equation}\label{closedcorr}
\text{\rm Corr}(\mathcal M_1^{\text{\rm cl}}(\alpha;\mathcal J);\text{\rm tri},ev^{\text{\rm int}})(1)
\in \Lambda(M).
\end{equation}
(\ref{closedcorr}) is a smooth differential form of degree
$$
\dim_{\R}M - \dim_{\R}\mathcal M_1^{\text{\rm cl}}(\alpha;\mathcal J)
= 6 - (6+ c^1(M) \cap [\alpha] + 2 - 6 + 1) = 3.
$$ 
\begin{dfn} We put:
$$
n(L;\alpha;\mathcal J) 
= \int_L \text{\rm Corr}(\mathcal M_1^{\text{\rm cl}}(\alpha;\mathcal J);\text{\rm tri},ev^{\text{\rm int}})(1)
\in \R.
$$
We also define:
$$
n(L;\alpha;\mathcal J;t) 
= \int_L \text{\rm Corr}(\mathcal M_1^{\text{\rm cl}}(\alpha;\mathcal J)
\cap ev_t^{-1}([0,t]);\text{\rm tri},ev^{\text{\rm int}})(1)
\in \R.
$$
\end{dfn}
The submersivity of $(ev_t,ev^{\text{\rm int}})$ implies that 
$n(L;\alpha;\mathcal J;t)$ is a smooth function of $t$.
\begin{thm}\label{thm61}
In the situation of Theorem \ref{wallcross}, 
$(\Lambda(L),\langle\cdot\rangle,\{\frak m_{k,\beta}^{J_0}\},\{\frak m_{-1,\beta}^{J_0}\})$
is pseudo-isotopic to 
$(\Lambda(L),\langle\cdot\rangle,\{\frak m_{k,\beta}^{J_1}\},\{\frak m_{-1,\beta}^{J_1} + \Delta(\beta)\})$
as inhomogeneous gapped cyclic filtered $A_{\infty}$ algebras.
Here
$$
\Delta(\beta) = \sum_{\tilde{\beta}: i_*(\tilde{\beta}) = \beta}n(L;\tilde{\beta};\mathcal J).
$$
\end{thm}
\begin{proof}
We consider the moduli space 
(\ref{famidiscmod}) and evaluation map
$$
(ev_t,ev) = (ev_t,ev_0,\ldots,ev_{k-1}) 
: \mathcal M_k(\beta;\mathcal J) \to [0,1] \times L^k.
$$
By \cite{F1} Section 11 we have a system of Kuranishi structures and 
families of multisections on $\mathcal M_k(\beta;\mathcal J)$ 
for $\beta \cap \omega < E_0$, with the following properties:
\par\medskip
\begin{enumerate}
\item The families of multisections are transversal to $0$.
\item They are compatible with the forgetful map 
\begin{equation}
\frak{forget}_{k,0} : \mathcal M_k(\beta;\mathcal J)
\to \mathcal M_0(\beta;\mathcal J).
\end{equation}
\item For $k \ge 1$ the evaluation map 
$$
(ev_t,ev_0) : \mathcal M_k(\beta;\mathcal J) \to [0,1] \times L
$$
is weakly submersive and
induces a submersion of the zero 
set of family of multisections, in the sense of \cite{F1} Definition 4.1.4. 
\item They are invariant under the cyclic permutation of the 
boundary marked points.
\item They are compatible with the identification (\ref{bdcompati0}).
\item We consider the decomposition:
\begin{equation}\label{paraJdecomp}
\aligned
\partial\mathcal M_{0}(\beta;\mathcal J) = &\bigcup_{\beta_1+\beta_2=\beta}
\left(\mathcal M_{1}(\beta_1;\mathcal J)\, {}_{(ev_t,ev_0)} \times_{(ev_t,ev_0)} \mathcal M_{1}(\beta_2;\mathcal J)
\right)/\Z_2 \\
&\cup \bigcup_{t\in [0,1]}\bigcup_{\tilde{\beta}:i_*(\tilde{\beta}) = \beta}
\{t\} \times \left(
\mathcal M_1^{\text{cl}}(\tilde{\beta};J_t) {}_{ev_0}\times_M L
\right).
\endaligned
\end{equation}
Then the Kuranishi structures and the families of multisections are 
compatible with (\ref{paraJdecomp}).
We use the Kuranishi structure and families of multisections on  $\mathcal M_1^{\text{cl}}(\tilde{\beta};J_t)$ 
which is explained in this section 
for the second term of the right hand side of (\ref{paraJdecomp}).
\item The evaluation map,
$
ev_t : \mathcal M_0(\beta;\mathcal J) \to [0,1]
$
is  weakly submersive and
 induces a submersion of the zero 
set of family of multisections, in the sense of \cite{F1} Definition 4.1.4. 
\item 
At $t_0=0,1$ the induced Kuranishi structure and families of multisecitons on 
$\mathcal M_k(\beta;\mathcal J) \cap ev_t^{-1}(\{t_0\})$ coincides with  
given choices Kuranishi structures and families of multisecitons on 
$\mathcal M_k(\beta;J_{t_0})$.
\end{enumerate}
\par\smallskip
This is mostly the same as one we used in the proof of Proposition
\ref{pseudohomoinvmodT}. The only difference is the second term of (\ref{paraJdecomp}).
It appears since the fiber product 
$\mathcal M_1^{\text{cl}}(\tilde{\beta};\mathcal J) {}_{ev_0}\times_M L$ 
can be nonempty in the situation where we consider one parameter family of complex structures.
\par
We now define $\frak m_{k,\beta}^t$, $\frak c_{k,\beta}^t$ for $k\ge 0$ in the same way 
as (\ref{46}), (\ref{pseudoisoshiki}) using $\mathcal M_k(\beta;\mathcal J)$ 
in place of $[0,1] \times \mathcal M_k(\beta;J)$.
\par
We finally define $\frak m^t_{-1,\beta}$ as follows.
We put:
\begin{equation}
\text{\rm Corr}_*(\mathcal M_0(\beta;\mathcal J);\text{\rm tri},ev_t)(1) 
=\rho(t) + dt \wedge \sigma(t)
\end{equation}
and define
\begin{equation}
\frak m^t_{-1,\beta} = \rho(t) + \sum_{\tilde{\beta}: i_*(\tilde{\beta}) = \beta} 
n(L;\tilde{\beta};\mathcal J;t).
\end{equation}
We can prove 
$(\Lambda(L),\langle\cdot\rangle,\{\frak m_{k,\beta}^{t}\},
\{\frak c_{k,\beta}^{t}\})$
is a pseudo-isotopy of gapped cyclic filtered $A_{\infty}$ algebra mod $T^{E_0}$
in the same way as \cite{F1} Section 11.
\par
To prove 
$(\Lambda(L),\langle\cdot\rangle,\{\frak m_{k,\beta}^{t}\},
\{\frak c_{k,\beta}^{t}\},\{\frak m_{-1,\beta}^{t}\})$
is an {\it inhomogeneous} pseudo-isotopy of gapped cyclic filtered $A_{\infty}$ algebra mod $T^{E_0}$
it suffices to prove (\ref{dervativemt}).
Let $0 \le t_1 < t_2 \le 1$. We have:
\begin{equation}\label{paraJdecomp2}
\aligned
&\partial\left(\mathcal M_{0}(\beta;\mathcal J) \cap ev_t^{-1}([t_1,t_2]) \right)\\
=& \left(\{ t_1 \} \times \mathcal M_{0}(\beta;J_{t_1})\right) 
\cup \left(\{ t_2 \} \times \mathcal M_{0}(\beta;J_{t_2}) \right) \\
&\cup \bigcup_{\beta_1+\beta_2=\beta}
\frac{\left(\mathcal M_{1}(\beta_1;\mathcal J)\, {}_{(ev_t,ev_0)} \times_{(ev_t,ev_0)} \mathcal M_{1}(\beta_2;\mathcal J)
\right)\cap ev_t^{-1}([t_1,t_2])}{\Z_2} \\
&\cup \bigcup_{t\in [t_1,t_2]}\bigcup_{\tilde{\beta}: i_*(\tilde{\beta}) = \beta}
\{t\} \times \left(
\mathcal M_1^{\text{cl}}(\tilde{\beta};J_t) {}_{ev_0}\times_M L
\right).
\endaligned
\end{equation}
We apply Stokes' theorem (\cite{F1} Proposition 4.2) to obtain:
\begin{equation}\label{formulaform}
\frak m_{-1,\beta}^{t_2} - \frak m_{-1,\beta}^{t_1}
= \sum_{\beta_1+\beta_2 = \beta}
\int_{t_1}^{t_2} \langle \frak c_{0,\beta_1}^t(1), 
\frak m_{0,\beta_2}^t(1) \rangle dt.
\end{equation}
Here the sum of the 1st and 3rd terms of (\ref{paraJdecomp2}) 
gives the left hand side of (\ref{formulaform}).
\par 
We obtain (\ref{dervativemt}) by differentiating (\ref{formulaform}).
\par
We remark 
$$
\frak m_{-1,\beta}^1 = \frak m_{-1,\beta}^{J_1} + 
\sum_{\tilde{\beta}: i_*(\tilde{\beta}) = \beta} n(L;\tilde{\beta};J).
$$
The proof of Theorem \ref{thm61} is complete.
(Actually we need to go from modulo $T^{E_0}$ version 
to Theorem \ref{thm61} itself. We omit this part since it is 
the same as one for Theorems \ref{superpotential} and \ref{inhomoexistmain}.)
\end{proof}
\section{Convergence}
\label{conv}
In this section we prove Theorem \ref{convmain}.
Actually most of the ideas of the proof is in \cite{F1} Section 13.
Let $b = \sum_{i=1}^{b_1} x_i \text{\bf e}_i$, where $\text{\bf e}_i$ 
is a basis of $H^1(L;\R)$. We put
$y_i = e^{x_i}$. For $\beta \in H_2(X,L;\Z)$ we define $\partial_i\beta \in \Z$ by
$\partial\beta = \sum_{i=1}^{b_1}\partial_i\beta \text{\bf e}_i$ and define
\begin{equation}
y^{\partial \beta} = \prod_{i=1}^{b_1} y_i^{\partial_i\beta}.
\end{equation}
\begin{thm}\label{strPsi}
We regard the superpotential $\Psi(b;J)$ as a function of $x_i$ then we have:
\begin{equation}
\Psi(b;J) = \sum_{\beta \in G} T^{\beta\cap[\omega]} \frak m_{-1,\beta}^J y^{\partial \beta}.
\end{equation}
\end{thm}
Theorem \ref{convmain}.1 follows immediately from Theorem \ref{strPsi}.
\begin{proof}
Let $\rho$ be a closed one form on $L$. By definition we have
$$\aligned
&\langle \frak m_{k,\beta}^J(\rho,\ldots,\rho),\rho\rangle \\
&= \text{\rm Corr}(\mathcal M_k(\beta;J);(ev_1,\ldots,ev_k,ev_0),\text{\rm tri})(\rho\times\cdots 
\times \rho)
\in \Lambda^0(\text{\rm pt}) = \R.
\endaligned$$
Then, by the same argument as the proof of Lemma 13.1 \cite{F1}, we have
$$
\langle \frak m_{k,\beta}^J(\rho,\ldots,\rho),\rho\rangle 
= \frac{1}{k!}(\rho \cap \partial \beta)^{k+1}\frak m_{-1,\beta}^J.
$$
Theorem \ref{strPsi} follows easily.
\end{proof}
\par\smallskip
We turn to the proof of Theorem \ref{convmain}.2.
We take a Weinstein neighborhood $U$ of $L$.
Namely $U$ is symplectomorphic to a neighborhood $U'$ of zero section 
in $T^*L$. 
We choose $\delta_1$ so that for $c = (c_1,\ldots,c_{b_1}) \in [-\delta_1,+\delta_1]^{b_1}$ 
the graph of the closed one form $\sum_{i=1}^{b_1} c_i \text{\bf e}_i$ is contained in $U'$. 
We send it by the symplectomorphism to $U$ and denote it by 
$L(c)$.
We may take $\delta_2 < \delta_1$ such that if 
$c = (c_1,\ldots,c_{b_1}) \in [-\delta_2,+\delta_2]^{b_1}$ then there exists a diffeomorphism 
$F_{c} : M \to M$ such that 
\begin{eqnarray}
&&F_c(L) = L(c), \label{73}\\
&&\text{$(F_c)_*J$ is tamed by $\omega$.} \label{74}
\end{eqnarray}
Then we have an isomorphism
\begin{equation}\label{LcandL}
\mathcal M_0(L(c);(F_c^{-1})^*(\beta),(F_c)_*J)
\cong 
\mathcal M_0(L;\beta,J).
\end{equation}
We can extend this isomorphism to their Kuranishi structures and 
family of multisections on them. We can then use Proposition \ref{canmaincanprp}
and (\ref{LcandL})
to obtain:
\begin{equation}
\frak m^J_{-1,\beta,L} = \frak m^{(F_c)_*J}_{-1,\beta,L(c)}.
\end{equation}
Here we include $L$ and $L(c)$ in the notation 
to clarify the Lagrangian submanifold we study.
Theorem \ref{strPsi} and \cite{F1} Lemma 13.5 then implies:
\begin{equation}\label{suppotidentity}
\Psi(y;L(c);(F_c)_*(J)) = \Psi(y(c);L;J),
\end{equation}
where we put
$
y(c)_i = T^{-c_i\partial_i\beta}y_i.
$
In (\ref{suppotidentity}) we include $L$ in the notation of superpotential 
to clarify the Lagrangian submanifold we study. 
We regard superpotential as a function of $y_i$ by using Theorem \ref{strPsi}.
\par
Since the right hand side converges in 
$\Lambda\langle\!\langle y_1,\ldots,y_{b_1},y_1^{-1},\ldots,y_{b_1}^{-1}\rangle\!\rangle$,
it follows that $\Psi(y(c);L;J)$ converges for $c=(c_1,\ldots,c_{b_1})$ with $\vert c_i\vert < \delta$.
This implies Theorem \ref{strPsi}.2.
\par
3 and 4 of Theorem \ref{strPsi} follows from Theorem \ref{superpotential}.
The proof of Theorem \ref{strPsi} is complete.
\qed
\par\medskip
Once the convergence is established, Propositions \ref{prp21}, \ref{gginvariance} and Theorems 
\ref{pseudoisowelldef}, \ref{inhomoexistmain}, \ref{canmainth} 
are generalized in the same way to our larger domain of convergence.
\section{Concluding remarks}
\label{DTawx}
\subsection{Rationality and integrality}
\label{Ratsubsec}
\medskip
\begin{conj}\label{conj1}
In the situation of Corollary \ref{maincor} we have
$\Psi^{\text{\rm can}}(0;J) \in \Q$.
\end{conj}
We remark that filtered $A_{\infty}$ structure on $H(L)$ is 
constructed in \cite{FOOO080} over $\Lambda_{0,nov}^{\Q}$.
In \cite{F1} and in this paper, we work over $\R$ coefficient 
to use continuous family of multisections and de Rham theory 
for construction.
This is the reason why we can not prove Conjecture \ref{conj1} 
by the method of this paper.
\begin{conj}\label{conj2} There exist integers
$\frak o_{\beta}^J \in \Z$ for each $\beta \in H_2(M,L;\Z)$ such that
\begin{equation}
\Psi^{\text{\rm can}}(0;J)  
= \sum_{d \in \Z_{+} : \beta/d \in H_2(M;L;\Z)} d^{-2}\frak o_{\beta/d}^J.
\end{equation}
\end{conj}
This is an anolog of the corresponding conjecture for 
Gromov-Witten invariant of genus zero. (See \cite{GV}.)
The factor $d^{-2}$ is discussed in \cite{Liu02}.
\subsection{Bulk deformation and generalization to
non Calabi-Yau case etc.}
\label{bulk}
In this paper we assumed $\dim_{\C} M = 3$, $c^1(M) = 0$, 
$\mu_L =0$. This assumption is used to 
define $\frak m_{-1,\beta}^J$.
Namely it is used to show that the 
(virtual) dimension of $\mathcal M_0(\beta;L;J)$ is $0$.
We may use bulk deformation 
(\cite{FOOO080} Section 3.8) to obtain a 
numerical invariant in some other cases, as follows.
\par
We consider the moduli space $\mathcal M_{\ell,k}(\beta;L;J)$ of 
bordered stable $J$-holomorphic curve of genus zero with $\ell$ interior 
marked points and $k$ boundary marked points, 
one boundary component and
of homology class $\beta$.
Let $\sigma_1,\ldots,\sigma_{\ell}$ be closed forms on $M$.
We may consider
$$
\text{\rm Corr}(\mathcal M_{\ell,0}(L;\beta;J);(ev^{\text{\rm  int}},\text{\rm tri}))
(\sigma_1,\ldots,\sigma_{\ell})
\in \Lambda^*(\text{\rm pt}) = \R
$$
if
$
* = n + \mu(\beta) - 3 + 2\ell - \sum \deg\sigma_i = 0. 
$
\par
We obtain similar numbers by considering $\mathcal M_{\ell,k}(L;\beta;J)$ 
and differential forms on $L$.
The algebraic structure behind this `invariant' is not 
yet clear to the author.  So the study of them is a problem for future research.
Another case where numerical invariant is defined 
is the case when $M$ is a toric manifold and $L$ is its $T^n$ orbit.
In that case $\mathcal M_{\ell,k}(L;\beta;J)$  of $\beta \in H_2(M,L;\Z)$ with Maslov index $\ge 2$ only 
is related to our structures.
See \cite{FOOO08I} and references therein for this case.
\subsection{The case of real point}
\label{real}
We assume $\dim_{\C} M = 3$ and let $\tau : M \to M$ be 
$J$-anti holomorphic involution. We assume that $L = \{ x \in M \mid \tau(x) = x\}$ is nonempty. 
Then it becomes a Lagrangian submanifold.
We assume $L$ is $\tau$-relatively spin 
(See \cite{FOOO06} Chapter 8 for its definition.
\cite{FOOO06} Chapter 8 will become \cite{FOOO09I}.) (If $L$ is spin then it is $\tau$-relatively spin.)
Then in \cite{FOOO06} Chapter 8 Sections 34 and 38, we constructed $\frak m^J_{k,\beta}$
such that
\begin{equation}\label{tausymmetry}
\frak m_{k,\tau_*(\beta)}^J(x_1,\ldots,x_k)
= (-1)^{k+1+*}\frak m_{k,\beta}^J
(x_k,\ldots,x_1)
\end{equation}
where
$* = \sum_{0\le i<j \le k} \deg'x_i\deg'x_j$.
(\cite{FOOO09I} Theorem 34.20.)
We can combine the construction of \cite{FOOO09I} with one in \cite{F1} 
and can define inhomogeneous cyclic filtered $A_{\infty}$
algebra $(\Lambda(L),\langle \cdot \rangle,\{\frak m_{k,\beta}^J\},\{\frak m_{-1,\beta}^J\})$
satisfying (\ref{tausymmetry}). Moreover $\frak m_{-1,\beta}^J$ satisfies
\begin{equation}
\frak m_{-1,\tau_*(\beta)}^J
= \frak m_{-1,\beta}^J.
\end{equation}
Then its superpotential satisfies
\begin{equation}
\Psi(-b;L;J) = \Psi(b;L;J).
\end{equation}
In particular $b=0$ is a critical point.
\begin{conj}
The critial value $\Psi(0;L;J)$ is equivalent to a particular case of the 
invariant by Solomon \cite{Solo}.
\end{conj}
We can prove $\Psi(0;L;J_0) = \Psi(0;L;J_1)$ if there exists a family of 
almost complex structures $J_t$ such that $\tau_*J_t = - J_t$.
In fact we can show
\begin{equation}
n(\tilde{\beta};L;\mathcal J) = - n(\tau_*\tilde{\beta};L;\mathcal J)
\end{equation}
for such $\mathcal J = \{J_t\}$.
\par
If we can generalize this construction in a way suggested in Subsection 
\ref{bulk}, it seems likely that we can reproduce the invariants of 
Solomon and Welschinger \cite{Wel}. 
\par
The superpotential we defined in this paper is also likely to be related to the  numbers studied by 
Walcher \cite{Wal}.
(For such a purpose we need to include flat bundle on $L$.
In fact in \cite{Wal} it seems that several flat connections are used to cancel the 
wall crossing term which appears in (\ref{wallcrossformula}).)
\subsection{Generalization to higher genus and Chern-Simons  perturbation theory}
\label{cs}
The right hand side of the formula (\ref{sumtreeformula}) has obvious similarity with the 
invariant of Chern-Simons perturbation theory (\cite{AxSi91II}).
It seems very likely that we can combine two stories to obtain 
an invariant counting the number of stable maps from bordered Riemann surface 
with arbitrary many boundary components and of arbitrary genus.
Its rigorous definition is not know at the time of writing of this paper.
The author is unable to do it at the time of writing of this paper because 
of the transversality problem. Here we describe some ideas and 
explain the difficulty to make it rigorous.
\par
Let $T$ be a ribbon graph. Namely it is a graph together with a choice of 
cyclic order of the sets of edges containing each vertices.
It uniquely determines a compact oriented 2 dimensional manifold $\Sigma(T)$ 
without boundary and an embedding $i : T \to \Sigma(T)$
such that the cyclic order of the edges are induced by the orientation 
of $\Sigma(T)$ and that the connected component of the complement 
$\Sigma(T) \setminus T$ are all discs.
(We do not assume that $T$ or $\Sigma(T)$ is connected.)
\par
Let $C_0(T)$ be the set of vertices and let $\ell = \#C_0(T)$.
For $v_i \in C_0(T)$, let $k_i$ be the number of edges containing $v_i$.
Let $e_{i,1},\ldots,e_{i,k_i}$ be the set of such edges.
The set of the pair $(v_i,e_{i,j})$ where $i=1,\ldots,\ell$, $j=1,\ldots,k_i$ 
is called a flag. Let $\text{\rm Fl}(T)$ be the set of flags. 
\par
We next consider a compact oriented 2 dimensional manifold $\Sigma$ 
with boundary $\partial \Sigma$. We assume $\partial \Sigma$ has 
at least $\ell$ connected components $\partial_i \Sigma$, $i=1,\ldots,\ell$ 
and on $\partial_i \Sigma$ we put $k_i$ boundary marked points.
There may be other component of $\partial \Sigma$, on which we do not
put boundary marked points.
(We remark that we do not assume that $\Sigma$ is connected.)
Each of the boundary marked points thus corresponds to an element 
of $\text{\rm Fl}(T)$.
\par
Let $\beta \in H_2(M,L;\Z)$ where $M$ is a 6 dimensional symplectic manifold 
with $c^1(M) = 0$ and $L$ its Lagrangian submanifold such that 
$H^1(L;\Q) =0$.
We consider the pair $(j,v)$ where $j$ is a complex structure on $\Sigma$ and 
$v : (\Sigma,\partial\Sigma) \to (M,L)$ is a $j-J$ holomorphic map.
Let $\mathcal M(\Sigma;\beta;L;J)$ be the moduli space of such pair.
(We take stable map compactification. It has a Kuranishi structure of 
dimension $\#\text{\rm Fl}(T)$.)
Evaluation map at each boundary marked points gives
\begin{equation}
ev : \mathcal M(\Sigma;\beta;L;J) \to L^{\#\text{\rm Fl}(T)}.
\end{equation}
We next consider the operator $G : \Lambda(L) \to \Lambda(L)$ of degree 
$+1$ as in Lemma 10.1 \cite{F1}. 
We can associate a distributional form $\tilde G$ on $L\times L$ or degree $2$
such that
$$
\langle G(u),v\rangle = \int \tilde G \wedge (u\times v).
$$
(See \cite{AxSi91II}.)
For each edge $e$ of $T$ we have
$\pi_e : L^{\#\text{\rm Fl}(T)} \to L^2$, that is the projection 
to the factors corresponding to $(v,e)$, $(v',e)$ where $\partial e = \{v,v'\}$.
We now `define'
\begin{equation}\label{nondef}
\frak m(T;\Sigma;\beta;L;J) =
\int_{\mathcal M(\Sigma;\beta;L;J)} ev^*
\left(
\prod_{e \in C_1(T)} \pi_e^*(\tilde G)
\right).
\end{equation}
To define the right hand side of (\ref{nondef}) rigorously, 
we need to take an appropriate perturbation of our 
moduli space $\mathcal M(\Sigma;\beta;L;J)$ and 
use it to define its virtual fundamental chain.
\par
The case when the genus of $\Sigma$ is $0$, 
$\Sigma$ has only one boundary component, and 
$T$ is a tree, is worked out in this paper and \cite{F1}.
In that case, it is important to find a perturbation so that it is 
compatible with the process to forget boundary marked points.
As we remarked in \cite{F1} Remark 3.2, the way we constructed such a 
continuous 
family of multisections in this paper and in  \cite{F1} uses the fact that 
the genus of $\Sigma$ is $0$. So it cannot be directly generalized to 
higher genus case.
\par
If we can find appropriate way to rigorously define 
(\ref{nondef}), we then put
\begin{equation}
\Psi(S,T;L;J)
= \sum_{T,\Sigma} S^{2-\chi(\Sigma \# \Sigma(T))} T^{\beta\cap [\omega]}
\frak m(T;\Sigma;\beta;L;J).
\end{equation}
This is 
expected to become an invariant of  $M,L,J$.
\par
Here $\Sigma \# \Sigma(T)$ is defined as follows.
For each $v \in C_0(T)$ we remove a small ball $B(v)$ centered at $v$
from $\Sigma(T)$. 
We then glue $\partial B(v_i)$ with the $i$-th boundary component 
of $\Sigma$. We thus obtain $\Sigma \# \Sigma(T)$ which is a compact 
oriented 2 dimensional manifold with or without 
boundary. $\chi(\Sigma \# \Sigma(T))$ is its Euler number.
We take the sum for $T,\Sigma$ such that $\Sigma \# \Sigma(T)$ is connected.
(Here the sum is over topological type of $\Sigma$ and $T$.
We actually need to divide each term by the order of appropriate 
automorphism group in a way similar to (\ref{sumtreeformula}).)
$S$ is a formal parameter which is called string coupling constant in physics
literature.
\begin{prob}
Let $M,L,J$ be a triple of symplectic manifold $M$, its Lagrangian submanifold $L$, and 
its tame almost complex structure $J$, such that $\dim L=3$, $c^1(M) = 0 = \mu_L$, 
$H^1(L;\Q) =0$.
Define an invariant $\Psi(S,T;L;J)$ such that 
at $T=0$ it becomes perturbative Chern-Simons invariant 
and at $S=0$ it becomes the invariant of 
Corollary \ref{maincor}.
\end{prob}
\begin{rem}
The study of Chern-Simons perturbation theory suggests 
that we need to fix framing of $L$ in order to obtain 
an appropriate perturbation.
\end{rem}
When we generalize the story to the case $H^1(L;\Q) \ne 0$, 
we need to consider the case when $T$ has exterior vertices 
and $\Sigma$ has a boundary marked point on the 
component other than $k_i$ components $\partial_i\Sigma$.
In that case we expect to obtain certain algebraic structure 
on $H^1(L;\Lambda_0)$. 
We believe that involutive-bi-Lie infinity structure 
(\cite{CFL}) is appropriate for this purpose.
More precisely this is the case when 
at least one element of $H^1(L;\Q)$ is assigned 
to each of the connected component of the boundary.
(In genus $0$ it corresponds to $\frak m_{k,\beta}$ with $k\ge 0$.)
If we restrict to such cases, the wall crossing phenomenon (the $J$ dependence) 
does not seem to occur. Namely the algebraic structure is expected to be independent 
of $J$ up to homotopy equivalence.
(This is certainly the case of genus zero as is proved in \cite{F1}.)
\subsection{Mirror to Donaldson-Thomas invariant}
\label{DT}
Let $M$ be a symplectic manifold of dimension $6$ and $c^1(M) =0$. 
We consider the set $\widetilde{\frak{Lag}}(M)$ of paris $(L,[b])$ such that
$L$ is a relatively spin Lagrangian submanifold with $\mu_L = 0$ and 
$[b] \in \mathcal M(L;\Lambda_0)$. 
\par
We say $(L,[b]) \sim (L',[b'])$ if there exists a Hamiltonian 
diffeomorphism $F : M \to M$ such that $L' = F(L)$ and 
$F_*(b)$ is Gauge equivalent to $b'$.
Let ${\frak{Lag}}(M)$ be the quotient space.
The quotient topology on ${\frak{Lag}}(M)$ is rather pathological. 
Namely it is likely to be non-Hausdorff in general.
We also need to take appropriate compactification 
of this moduli space by including singular Lagrangian submanifolds, 
for example.
(Such a compactification is not known at the time of writing this 
paper.)
\par
On the other hand, we can define a `local chart' of 
${\frak{Lag}}(M)$ as follows. Let $(L,[b_0]) \in {\frak{Lag}}(M)$.
We take $\delta > 0$ small such that for $L(c)$ with 
$c = \sum c_i\text{\bf e}_i$, $\vert c_i\vert < \delta$, there exists $F_c$ as in (\ref{73}),(\ref{74}).
We consider 
$$
A(\delta) = \{(y_1,\ldots,y_{b_1}) \mid y_i \in \Lambda, \vert v(y_i)\vert < 
\delta\}.
$$
Then a neighborhood of $(L,[b_0])$ is identified with the set of $(y_1\ldots,y_m) \in A(\delta)$ 
satisfying the Maurer-Cartan equation
\begin{equation}\label{MCeqs}
\sum_{k=0}^{\infty}\sum_{\beta} \frak m_{k,\beta}(y_1,\ldots,y_m) =0.
\end{equation}
We remark the equation (\ref{MCeqs}) is well defined by Theorem \ref{convmain}.
\par
$y_i = e^{x_i} = T^{c_i}y'_i$ with 
$c_i = v(y_i)$ then $b' = \sum \log y'_i \text{\bf e}_i$ and $L(c)$ defines 
an element of $\frak{Lag}(M;(F_c)*J)$. 
(See Section \ref{conv} and \cite{F1} Section 13.)
Using the independence of 
Maurer-Cartan scheme of almost complex structure, we obtain an 
element of  $\frak{Lag}(M) = \frak{Lag}(M;J)$.
Thus, one may regard ${\frak{Lag}}(M)$ as a kind of  `non-separable rigid analytic stack'.
\par
We remark that the equation (\ref{MCeqs}) is equivalent to
$$
\nabla_y \Psi = 0.
$$
\par
Thus our situation is similar to one which appears in Donaldson-Thomas 
invariant. (Thomas \cite{Thomas}, Joyce \cite{Joy}, Kontsevich-Soibelman \cite{KS}.)
There the role of superpotential is taken by the holomorphic Chern-Simons
invariant.
\begin{prob}
\begin{enumerate}
\item Find an appropriate stability condition for the pair $(L,[b])$
and use it to construct a moduli space ${\frak{Lag}}^{\text{\rm st}}(M)$ 
of stable pairs $(L,[b])$ which 
has better properties than ${\frak{Lag}}(M)$.
\item Define an invariant which is the `order' of ${\frak{Lag}}^{\text{\rm st}}(M)$ 
in the sense of virtual fundamental cycle.
\item Prove that it coincides with Donaldson-Thomas invariant of the 
Mirror manifold of $M$.
\end{enumerate}
\end{prob}
It seems to the author that this problem is very difficult to study at this stage.
\par\medskip
\begin{rem}
After \cite{F1} had been put on an arXiv, and at the time of 
final stage of writing this article, 
a paper \cite{vitto} was put on an arXiv, where a 
different construction of a similar invariant as one in Corollary \ref{maincor}
(over $\Q$)
is sketched.
\end{rem}
{\scshape 
\begin{flushright}
\begin{tabular}{l}
Kyoto University \\
Kitashirakawa, Sakyo-ku, \\
Kyoto 602-8502, \\
Japan \\
{\upshape e-mail: fukaya@math.kyoto-u.ac.jp}\\
\end{tabular}
\end{flushright}
}

\end{document}